\documentclass[11pt,leqno]{amsart}

\usepackage{amsthm,amssymb,amsmath,latexsym}
\usepackage{amsmath}
\usepackage{amssymb}
\usepackage{latexsym}
\usepackage{graphics}
\usepackage{graphicx}

\newtheorem{theorem}{Theorem}[section]
\newtheorem{lemma}[theorem]{Lemma}
\newtheorem{proposition}[theorem]{Proposition}
\newtheorem{corollary}[theorem]{Corollary}
\numberwithin{equation}{section}

\theoremstyle{definition}
\newtheorem{definition}[theorem]{Definition}
\theoremstyle{remark}
\newtheorem{remark}[theorem]{Remark}
\theoremstyle{remark}
\newtheorem{example}[theorem]{Example}
\theoremstyle{remark}

\def\RR{{\mathbb R}}
\def\ZZ{{\mathbb Z}}
\def\Z{{\mathbb Z}}

\def\NN{{\mathbb N}}
\def\Nat{\NN}

\def\ZZ{{\mathbb Z}}

\def\B{{\mathcal B}}
 \def\C{{\mathcal C}}
 
\def\A{{\mathcal A}}
\def\T{{\mathcal T}}
\def\N{{\mathcal N}}
\def\G{{\mathcal G}}
\def\K{{\mathcal K}}

\def\R{{\RR}}
\def\Tk{{\scriptscriptstyle \T}}
\def\Sk{{\mathcal S}}

\def\supp{{\rm supp}}

\def\lam{\lambda}
\def\sig{\sigma}

\def\Anonp{\A_{\rm nonp}}

\def\Int{{\rm int}}

\def\Ak{\A}

\def\Bk{\mathcal B}
\def\ov{\overline}
\def\Gam{\Gamma}
\renewcommand{\o}{\overline{1}}
\def\be{{\bf e}}

\def\Tk{{\mathcal T}}
\def\b0{{\bf 0}}
\def\Dk{{\mathcal D}}
\def\Ek{{\mathcal E}}
\def\Gk{{\mathcal G}}
\def\Hk{{\mathcal H}}
\def\Om{\Omega}

\def\Lk{{\mathcal L}}
\def\beq{\begin{equation}}
\def\eeq{\end{equation}}
\def\mutr{\mu^{{\rm tr}}}
\def\gam{\gamma}

\begin{document}

\title{Second Order Ergodic Theorem for  Self-Similar Tiling Systems}

\author{Konstantin Medynets}
\address{Konstantin Medynets, Department of Mathematics, U.S. Naval Academy, Annapolis, MA 21402, USA}
\email{medynets@usna.edu}
\author{Boris Solomyak }
\address{Boris Solomyak, Department of Mathematics,
University of Washington, Seattle, WA 98195-4350, USA}
\email{solomyak@uw.edu}

  \begin{abstract}  We consider  infinite measure-preserving non-primitive self-similar tiling systems  in Euclidean space $\mathbb R^d$. We  establish the second-order ergodic theorem for such systems, with exponent equal to the Hausdorff dimension of a graph-directed self-similar set  associated with the substitution rule.
\end{abstract}

\date{\today}

\thanks{B. S.  is supported in part by NSF grant DMS-0968879.}

\maketitle

\section{Introduction}

Let $\mathbb X = (X,\mu,\mathbb T)$ be a conservative, ergodic, measure preserving dynamical system on a $\sigma$-finite measure space. If $f\in L^1(X,\mu)$, then  Hopf's ratio ergodic theorem says that the growth of $S_nf(x) = f(x)+\cdots + f(T^{n-1}x)$ is independent of $f$ in the sense that if $g\in L^1(X,\mu)$ with $\int g\, d\mu \neq 0$, then $$\frac{S_nf(x)}{S_ng(x)}\to \frac{\int f \,d\mu}{\int g\,d\mu}\mbox{ for }\mu\mbox{-a.e. }x\in X.$$
It turns out that due to the measure $\mu$ being infinite, it is impossible to replace functions $S_ng(x)$ by constants $\{a_n\}$ \cite[Theorem 2.4.2]{Aaronson:book}. However, it was observed earlier \cite{Fisher:1992,Aarson_Denker_Fisher:1992,Ledrappier_Sarig:2008} that for some systems, the ratios $S_nf(x)/a_n$ (for some choice of $a_n$) still converge to $\int f\,d\mu$,  though in a weaker  sense (second-order averages).  The asymptotic behavior of the sequence  $\{a_n\}$ is an invariant of the dynamical system.

The main result of the present paper is  the following ergodic theorem showing that for self-similar tilings the sequence $\{a_n\}$ can be chosen as $\{n^{\alpha+1}\}$ where $\alpha$ is an intrinsic parameter of the system reflecting self-similarity (the precise statements, with all technical assumptions, are Theorem~\ref{TheoremMainResult} and Theorem~\ref{TheoremSecondOrderSubstitutions}).

\begin{theorem}\label{TheoremIntroMainResult} (i) Let $\mathbb X =  (\Omega,\mu,\R^d)$ be a non-primitive self-similar substitution tiling system preserving an infinite ergodic measure $\mu$. Assume that the measure $\mu$ is non-zero and finite on some open subset of $\Omega$.  Then there exist positive  parameters $\alpha$ and $c$ such that  for every function $f\in L^1(\Omega,\mu)$ and  $\mu$-almost every tiling $\mathcal T\in \Omega$,  we have
\begin{equation}\label{EquationIntroLogAverages}
\lim\limits_{t\to\infty}\frac{1}{\log(t)}\int_{1}^{t}\frac{\int_{B_R} f(\mathcal T - u)\,du}{c(2 R) ^{\alpha}}\,\frac{dR}{R}  = \int_\Omega f d\mu.
\end{equation}
Here  $B_R$ is the ball of radius $R$ centered at the origin.

(ii) The parameters $\alpha$ and $c$ are invariants of the measure-theoretic isomorphism of the system $\mathbb X$.

(iii) This result is also valid for a large class of one-dimensional symbolic substitution systems with  integrals in the left-hand side replaced by the corresponding sums.
\end{theorem}

Our work was originally inspired by A.~Fisher's paper \cite{Fisher:1992}, where he obtained a similar ergodic theorem for a single substitution system $(X_\sigma,S)$ generated by the map $\sigma(0) = 000$ and $\sigma(1) =  101$.  Iterating the map $\sigma^n(1)$, $n\geq 1$, we get a sequence where the appearances of 1's and 0's resemble the process of constructing the middle-thirds Cantor set.  Fisher  used this analogy to show that for any function $f\in L^1(X_\sigma,\mu)$ and $\mu$-a.e point $x\in X_\sigma$,
 $$\lim_{n\to\infty}\frac{1}{\log(n)} \sum_{k=1}^n\frac{\sum_{i=0}^{k-1}  f(S^i x)}{c k^\alpha}\cdot \frac{1}{k} = \int_{X_\sigma}fd\mu,$$
  where $\mu$ is the unique $S$-invariant measure $\mu$ on $X_\sigma$ with $\mu([1]) = 1$. Here $\alpha = \log(2)/\log(3)$ is the Hausdorff dimension and $c$ is the average density of the Hausdorff measure on the middle-thirds Cantor set. Average densities were introduced by Bedford and Fisher \cite{BeFi} (see Definition~\ref{def-avden}).  The middle-thirds Cantor set arising in the study of the substitution $\sigma$ is a special case of the self-similar set (in general, for a graph-directed iterated function system) that one can associate to every self-similar tiling substitution (Section \ref{SubsectionGraphDirectedIFS}). In the proof of  Theorem \ref{TheoremIntroMainResult}, we  show the parameters $\alpha$ and $c$ always arise as the Hausdorff dimension and the appropriate average density of the Hausdorff measure of the associated graph-directed sets.

%\begin{sloppypar}
It is well-known that primitive substitution dynamical systems (both symbolic and tiling versions) are uniquely ergodic. Non-primitive substitutions and their invariant measures have been recently
studied in \cite{Yuasa:2007,BezuglyiKwiatkowskiMedynets:2009,BezuglyiKwiatkowskiMedynetsSolomyak:2010,HamaYuasa:2011}, where it was shown, in particular, that for a large class of
non-primitive symbolic substitutions infinite ($\sigma$-finite) invariant measures appear naturally. This has also been extended to the tiling setting in \cite{CortezSolomyak:2011}. The present work  continues this line of research.
%focusing on ergodic properties of the resulting infinite measure-preserving transformations.\end{sloppypar}

We observe that the second-order ergodic theorem is not a universal result. Its validity depends on  intrinsic properties of  the dynamical system in question. We refer the reader  to the paper \cite{Aarson_Denker_Fisher:1992} of Aaronson-Denker-Fisher for the discussion of second order ergodic theorems  for  Markov shifts. We also mention  the paper \cite{Ledrappier_Sarig:2008}  of Ledrappier and Sarig establishing the second order ergodic theorem for certain horocycle flows.

The basic idea of the proof of Theorem \ref{TheoremIntroMainResult} has roots in the work of Bedford and Fisher \cite{BeFi}. In a nutshell, it follows from the Birkhoff ergodic theorem for the
renormalization, or ``scaling'' map, associated with the substitution, restricted to a certain fractal subset of the tiling space. The negative part (going into the past) of this renormalization map turns out to be essentially the time-1 map of the scenery flow arising from ``zooming in'' into the fractal.

The structure of the paper is the following. In Section \ref{SectionTilingSystems} we give the definition of the substitution tiling system $(\Omega,\R^d)$
 generated by a tile substitution $\mathcal G$. We then  recall the classification  of infinite ergodic invariant measures established in \cite{CortezSolomyak:2011}. In subsection \ref{SectionAssumptions}  we explicitly state technical assumptions on the tiling system needed for  Theorem \ref{TheoremIntroMainResult}.

In Section \ref{SectionTransverseDynamics} we show how to associate a graph-directed set to the substitution rule $\mathcal G$. In the
following sections we use this fractal to count frequencies of prototiles in tilings of $\Omega$.

In Section \ref{SectionTransverseMeasures}, we show that the transverse dynamical system (generated by iterations of the substitution rule) is measure-theoretically isomorphic to a  Markov chain.

Sections \ref{SectionSecondOrderErgodicTheorem} and \ref{SectionSubstitutionDynamicalSystems} are devoted to the proof of Theorem \ref{TheoremIntroMainResult}. As a corollary, we establish that almost every sequence from a substitution space admits a non-zero ``$\alpha$-dimensional frequency,'' where the parameter $\alpha$ comes from Theorem \ref{TheoremIntroMainResult}.
Section \ref{SectionExamples} contains a few examples and open questions.

%%%%%%%%%%%%%%%%%%%%%%%%%%%%%%%%%%%%%%%%
%
%

\section{Tiling Dynamical Systems}\label{SectionTilingSystems}

In this section we fix our notation and present  necessary definitions from  tiling dynamics. We mostly follow conventions of \cite{CortezSolomyak:2011}, see also \cite{Robi}.

\subsection{Tiling Space}

Fix a finite alphabet $\mathcal W$ and an integer $d\geq 1$.  By a {\it tile} in $\mathbb R^d$ we  mean a pair $T = (F,i)$  of a  compact set $F$  that is the closure of its interior, and  a letter (label) $i\in \mathcal W$.  Two geometrically identical sets labeled by different letters are treated as distinct tiles. The set $F$ will be called the {\it support} of the tile $T$, in symbols, $\textrm{supp}(T) = F$. A {\it tiling} is a family of tiles $\mathcal T$  such that $\mathbb R^d = \cup  \{\textrm{supp}(T) : T\in \mathcal T\}$ and distinct tiles have disjoint interiors.  A {\it patch}  is a finite set of tiles with disjoint interiors.  The {\it support of a patch $P$} is the set $\textrm{supp}(P) = \cup \{\textrm{supp}(T) : T\in P\}$. If $\mathcal T$ is a tiling, its finite subsets are called {\it $\mathcal T$-patches.}

The translate of a tile $T = (F,i)$ by a vector $u\in \mathbb R^d$ is the tile $T+ u =(F+u,i)$. Similarly, a translate of a patch $P$ by $u\in \mathbb R^d$ is the patch $P + u = \{T + u : T\in P\}$. We say that two patches $P_1$ and $P_2$ are {\it translationally equivalent} if $P_1 = P_2 +u$ for some vector $u\in \mathbb R^d$.

 \begin{definition} Let $\mathcal A$ be  a finite set of tiles in $\mathbb R^d$ such that distinct tiles from $\mathcal A$ are not translationally equivalent.  Tiles from the set $\mathcal A$ are called {\it prototiles}.  We will call translations of prototiles  {\it $\mathcal A$-tiles}.

 Denote by $\mathcal A^+$ the set of patches made of translates of tiles from $\mathcal A$.
  We assume that every prototile $T\in \mathcal A$ is centered at the origin, in the sense that $\textbf{0}\in\Int(\textrm{supp}(T))$.
% Each prototile can be viewed as ``pinpointed'' at the origin. For any tile $T '$, which is a translate of a prototile  $T\in \mathcal A$, denote by $p(T')$ the vector in $\mathbb %R^l$ with $T' - p(T') = T.$
%We will call $p(T')$ the {\em distinguished point} in $T'$.

Let $\varphi$ be an expanding linear transformation in $\R^d$. A map $\mathcal G: \mathcal A\rightarrow \mathcal A^+ $ is called a {\it tile substitution with expansion $\varphi$}  if \begin{equation}\label{EquationDefinitionTileSubstitution}
\textrm{supp}(\mathcal G(T)) = \varphi(\textrm{supp}(T))\mbox{ for every tile }T\in \mathcal A.
 \end{equation}
\end{definition}
In other words, the substitution $\mathcal G$ shows how to subdivide the inflated tile $\varphi(\textrm{supp}(T))$ into translates of prototiles.  The tile substitution can
be written explicitly as follows:
\beq \label{EquSub}
\Gk(T) = \bigcup_{T'\in \Ak} \{T'+u:\ u\in \Dk_{T,T'}\} \ \ \mbox{for all}\ \ T\in \Ak,
\eeq
where $\Dk_{T,T'}$ is a finite (possibly empty) subset of $\R^d$, the tiles in the right-hand side have disjoint interiors, and
\beq \label{EquSub2}
\varphi(\supp(T)) = \bigcup_{T'\in \Ak} \bigcup_{u\in \Dk_{T,T'}} (\supp(T')+u).
\eeq
The substitution $\mathcal G$ is extended to translates of prototiles by $\mathcal G(T + u) = \mathcal G(T) + \varphi(u)$; and to patches by $\mathcal G(P) = \cup \{\mathcal G(T) : T\in P\}$. The linearity of $\varphi$ and the equation (\ref{EquationDefinitionTileSubstitution}) imply that the patch $\mathcal G(P)$ is well-defined.

\begin{remark}\label{RemSelfsim}
In this paper we restrict ourselves to the {\em self-similar} case, i.e.\ $\varphi = \lam\cdot O$, where $O$ is an orthogonal matrix and $\lam>1$. We refer to the
corresponding $\G$ as {\em self-similar tiling substitution}. The more general case of an arbitrary expansion map $\varphi$, referred to as {\em self-affine}, is
not covered by our main results.
\end{remark}

\begin{sloppypar}
\begin{definition} For a given tiling substitution $\mathcal G: \mathcal A\rightarrow \mathcal A^+$, let $M_\Gk = (m_{A,B})_{A,B\in \mathcal A}$ be the matrix with $m_{A,B}$ being the number of translates of the prototile $A$ in the patch $\mathcal G(B)$ (i.e.\ $m_{A,B}=\# \Dk_{A,B}$). The matrix $M_\Gk$  is  called the {\it substitution matrix of $\mathcal G$}.
\end{definition}
\end{sloppypar}

The substitution is called {\em primitive} if some power of the substitution matrix has only positive entries. We emphasize that our focus is on the non-primitive case.

The following example will help  illustrate the concepts as we go along. We call it the  {\em integer Sierpi\'nski carpet tiling substitution}, by analogy with A. Fisher's integer Cantor set substitution
\cite{Fisher:1992}.

\begin{example} \label{ex-carpet}  Suppose that the prototile set $\mathcal A$  consists of two $1\times 1$ squares on the plane labeled by $0$ and $1$ (we will call them the ``0-tile'' and ``1-tile''). Consider the following tile substitution $\Gk$:

$$\begin{array}{|c|}
\hline
    0 \\
    \hline
  \end{array}
\mapsto
\begin{array}{|c|c|c|}
\hline
  0 & 0 & 0 \\
  \hline
  0 & 0 & 0 \\
  \hline
  0 & 0 & 0\\
  \hline
\end{array}
\qquad\mbox{  and  }\qquad
\begin{array}{|c|}
\hline
    1 \\
    \hline
  \end{array}
\mapsto
\begin{array}{|c|c|c|}
\hline
  1 & 1 & 1 \\
  \hline
  1 & 0 & 1 \\
  \hline
  1 & 1 & 1\\
  \hline
\end{array}
$$
The expansion map is a dilation: $\varphi = 3I$.
The substitution matrix $M_\Gk = \left(\begin{array}{cc} 9 & 1 \\ 0 & 8 \end{array} \right)$ is non-primitive.
 \end{example}

\begin{definition} \label{def-tilspace} Let $\mathcal G : \mathcal A\rightarrow \mathcal A^+$ be a tile substitution. Denote by $\Omega_{\mathcal G}$ the set of all tilings of $\mathbb R^d$ by tiles from $\mathcal A$ such that $\mathcal T \in \Omega_{\mathcal G}$ if every $\mathcal T$-patch is a subpatch of $\mathcal G^n(T) + u$ for some $T\in \mathcal A$, $u\in \mathbb R^d$, and $n\geq 1$. The set $\Omega_{\mathcal G}$ is called the {\it tiling space} corresponding to the substitution $\mathcal G$. The group $\mathbb R^d$ has a natural translation action on $\Omega_\mathcal G$ given by $ u:\mathcal T\mapsto \mathcal T-u$ for every $u\in \mathbb R^d$ and  $\mathcal T\in \Omega_\mathcal G$. The pair $(\Omega_\mathcal G,\mathbb R^d)$ is called a {\it substitution tiling system}.
\end{definition}

Let $\|\cdot \|$ be the Euclidean norm  on $\mathbb R^d$. For $x\in \mathbb R^d$ and $R>0$, set $B_R(x) = \{u\in \mathbb R^d:\, \|u-x\| \le R \}$. We will write $B_R$ for $B_R(\textbf{0})$. For a compact set $K$ and a tiling $\mathcal T$, denote by $\mathcal T[[K]]$ the set of all $\mathcal T$-patches $P$ with $K\subset \textrm{supp}(P)$.  Define a metric $\rho$  on the space $\Omega_\mathcal G$ as follows. Given tilings $\mathcal T',\mathcal T''\in\Omega_\mathcal G$,  let $\rho(\mathcal T',\mathcal T'')$ be the minimum of $2^{-1/2}$ and
$$ \inf\limits\{ r >0 : \exists g\in B_r,P'\in \mathcal T'[[B_{1/r}]], P''\in \mathcal T''[[B_{1/r}]]\mbox{ such that }P' -g = P'' \}.$$ With respect to the topology generated by $\rho$, two tilings are close to each other if they agree on a large ball around the origin after a small translation. The cut-off parameter $2^{-1/2}$ is needed to fulfill the triangle inequality.

\begin{definition} \label{def-FLC}

(1) We say that the tiling system $\Omega_{\mathcal G}$  has {\it finite local complexity} (FLC) if for every tiling $\mathcal T \in \Omega_{\mathcal G}$ and  $R >0$, there are only finitely many $\mathcal T$-patches of diameter less than $R$ up to translation equivalence. (Note that {\em finite pattern condition} and {\em translational finiteness} are sometimes used in the literature synonymously with FLC).

(2) The tiling substitution $\mathcal G$ is called {\it admissible} if for every prototile $T\in \mathcal A$ there exists a tile $\mathcal T\in \Omega_\mathcal G$ such that $T\in \mathcal T$.
\end{definition}

It is not always trivial  to verify these conditions. Of course,  $\Omega_\Gk$ in Example~\ref{ex-carpet} has FLC.
Less obvious examples of FLC tile substitutions (with tiles having fractal boundary) are considered e.g.\ in \cite{Kenyon:1996,Solomyak:1997}, and two of them are discussed in Section 7.
There exist primitive tile substitution systems that do not have the FLC property, see \cite[p.244]{Kenyon:1991} and \cite{Danzer,FraRobi} (the latter ones have polygonal tiles). 

Next let us verify that the tiling substitution in Example~\ref{ex-carpet} is admissible.
%Iterating the substitution of the tile labeled 0, we see that there is a tiling of all zeros in the space $\Omega_{\mathcal G}$ (in fact, there are infinitely many of them, depending on
%where we put the origin, so the set of such ``zero'' tilings forms a subspace homeomorphic to the 2-torus $\T^2$; we will see later that this is the unique minimal subset for the $\R^2$ action).
If we put a 1-tile  on the plane so that one of its corners is at the origin and start iterating the substitution, we will obtain an increasing sequence of patches which agree with each other and tend to a tiling of a quarter-plane. Since the second iterate $\Gk^2\left(\,\begin{array}{|c|}
\hline
    1 \\
    \hline
  \end{array}\,\right)$
  contains a patch of the form $\begin{array}{|c|c|}
\hline
  1 & 1  \\
  \hline
  1 & 1 \\
  \hline
\end{array}
$\,,
we can start with this ``seed'' centered at the origin and obtain a tiling of the entire plane, which is the union of the four quarter-plane tilings. This tiling is then in $\Om_\Gk$, which confirms admissibility.

\medskip

The FLC assumption implies the following result. The proof can be found, for example, in \cite[Lemma 2]{RadinWolff}.

\begin{proposition} If the tiling system has the FLC property,  then the set $\Omega_\mathcal G$ is compact with respect to the topology generated by the metric  $\rho$.  The action of the group $\mathbb R^d$ by translations on $\Omega_\mathcal G$ is continuous.
\end{proposition}

Along with the translation $\R^d$-action, we have the {\em substitution action} on $\Om_\Gk$, which is denoted by the same letter $\Gk:\, \Om_\Gk\to \Om_\Gk$. These two dynamical systems are
intertwined by the relation
\beq \label{eq-rela}
\Gk(\Tk-y) = \Gk(\Tk) - \varphi(y),\ \Tk\in \Om_\Gk,\ y\in \R^d.
\eeq

The following result shows that any tiling in $\Omega_\mathcal G$ has a preimage under the map $\mathcal G$, see \cite[Lemma 2.8]{CortezSolomyak:2011}.

\begin{proposition}\label{PropositionTilingSurjection} If $\mathcal G$ is admissible, then the map $\mathcal G : \Omega_\mathcal G \rightarrow \Omega_\mathcal G$ is a continuous surjection.
\end{proposition}

In fact, this is almost immediate. To find a pre-image of $\Tk\in \Om_\Gk$ under $\Gk$ one needs to ``compose'' or ``combine'' its tiles into patches that are translates of substituted prototiles (i.e.\ $\Gk(T),\ T\in \A$), so that the resulting tiling, after rescaling by $\varphi^{-1}$, belongs to $\Om_\Gk$. This is always possible locally, by the definition of $\Om_\Gk$ (Definition~\ref{def-tilspace}), and one only has
take a subsequential limit.

One of the important issues in the theory of substitutions is to understand when the map $\mathcal G$ is invertible. This property is sometimes referred to as {\em recognizability} or {\em unique composition}.  In fact, $\Tk\in \Om_\Gk$ has a unique pre-image under $\Gk$ whenever the ``composition''  described above is unique. Global invertibility of $\Gk$ is equivalent to non-periodicity of the tiling space for primitive tile substitutions \cite{Solomyak:1998}, but the extension to the non-primitive case is by no means trivial.

\begin{definition}\label{DefNonper1} A tiling $\mathcal T\in \Omega_\mathcal G$ is {\it periodic} if there is a non-zero vector $u\in \mathbb R^d$ such that $\mathcal T = \mathcal T + u$. Such a vector $u$ is called a {\em period} of $\Tk$. A tile substitution $\mathcal G$ is called {\it non-periodic} if the set $\Omega_\mathcal G$ has no periodic tilings.
\end{definition}

The set of periods of a tiling in $\R^d$ is a subgroup of $\R^d$.
Periodic tilings can be further classified according to the rank of the group of periods, but this will not concern us in this paper. A classical argument shows that $\Gk$ cannot be invertible if the
tiling space contains a periodic tiling. Indeed, if $\Gk(\Sk) = \Tk$ and $\Tk=\Tk+u$, then $\Gk(\Sk) = \Gk(\Sk + \varphi^{-1}u) = \Tk$ by (\ref{eq-rela}). If $\Sk = \Sk + \varphi^{-1}u$, we can repeat this, obtaining
shorter and shorter periods. However, this cannot go on indefinitely, since the period of a tiling cannot be shorter than the diameter of the largest ball which is contained in the interior of all the
prototiles.

Next we discuss minimal components of our tiling dynamical systems. Recall that a dynamical system is {\em minimal} if it has no proper closed invariant subsets. It is well-known that tiling dynamical
systems arising from primitive substitutions are minimal (see e.g.\ \cite{Robi}). 

 \begin{definition}\label{Defmincomp}
A {\it minimal component} of the system $(\Omega_\mathcal G, \mathbb R^d)$ is a closed $\mathbb R^d$-invariant set that contains no proper closed invariant subsets. We note that if a minimal component $\Omega$ contains a tiling with period $u\in \mathbb R^d$, then every tiling of $\Omega$ has the period $u$.
\end{definition}

An easy consequence of Proposition~\ref{PropositionTilingSurjection} is that $\Om_{\Gk} = \Om_{\Gk^k}$ for any $k\in \Nat$ (see \cite[Lemma 2.9]{CortezSolomyak:2011}), hence one can replace
the tiling substitution with its power, whenever convenient.
Reordering the letters in the alphabet $\mathcal A$ and replacing  $\mathcal G$ with its higher power $\mathcal G^k$ if needed, the substitution matrix can be reduced to the following form:

\begin{equation}\label{FrobeniusForm}
M_\Gk =\left(
  \begin{array}{ccccccc}
    F_1 & 0 & \cdots & 0 &  X_{1,s+1} & \cdots & X_{1,m} \\
    0 & F_2 & \cdots & 0 & X_{2,s+1} & \cdots & X_{2,m} \\
    \vdots & \vdots & \ddots & \vdots & \vdots & \cdots& \vdots \\
    0 & 0 & \cdots & F_s & X_{s,s+1} & \cdots & X_{s,m} \\
   0 & 0 & \cdots & 0 & F_{s+1} & \cdots & X_{s+1,m} \\
    \vdots & \vdots & \cdots & \vdots & \vdots & \ddots & \vdots \\
    0 & 0  & \cdots &  0 & 0 & \cdots & F_m
  \end{array}\right)
\end{equation}

\noindent The square  matrices $F_i$  on the main diagonal are either  zero matrices or contain only strictly positive entries.  For any fixed $j = s+1,..., m$, at least one of the matrices $X_{k,j}$ is non-zero.
The block-triangular form (\ref{FrobeniusForm}) allows us to give an effective description of minimal components of the  system.

For each $i=1,\ldots,m$,   denote by $\mathcal A_i$ the set of prototiles corresponding to the block $F_i$.  Denote by $\Omega_i$ the set of tilings $\mathcal T\in \Omega_\mathcal G$ whose patches are subpatches of $\mathcal G^n(T)$, $n\geq 0$, with $T\in \mathcal A_i$. For $i\le  s$ it follows from the structure of $M_\Gk$ that $\Gk(\Ak_i) \subset \Ak_i^+$, and then $\Om_i$ is just the substitution tiling system for the restriction of $\Gk$ to $\Ak_i$. Since $F_i$ is strictly positive for $i\le s$, we see that $(\Om_i,\R^d)$ is minimal. It is not hard to show  (\cite[Lemma 2.10]{CortezSolomyak:2011})  that the sets $\{\Omega_1,\ldots,\Omega_s\}$  are the only minimal components of the tiling system.

%We refer the reader to the papers \cite{BezuglyiKwiatkowskiMedynetsSolomyak:2010} and \cite{CortezSolomyak:2011} for more details on the structure of minimal components of substitution and tiling systems.

 In Example~\ref{ex-carpet}, the unique minimal component consists of the tilings with 0-tiles only. It is, of course, periodic, with the lattice $\Z^2$ being the group of periods, so it is topologically conjugate to the translation action on the 2-torus
 $\R^2/\Z^2$.

Next we state a sufficient condition for
the invertibility of the map $\mathcal G$, following \cite{CortezSolomyak:2011}, for which we need another definition.

\begin{definition}\label{DefNonper}
(1) Denote by $\mathcal A_{\textrm{per}}$ the set of all tiles $T\in \mathcal A$ that occur in periodic tilings from   minimal components.  Set
$\mathcal A_{\textrm{nonp}} = \mathcal A \setminus \mathcal A_{\textrm{per}}$. 
We emphasize that prototiles which do not appear in minimal components belong to $\mathcal A_{\textrm{nonp}}$ by default.

(2) A substitution $\mathcal G$ is said to {\it satisfy the non-periodic border condition} (NBC) if for every tile $T\in \mathcal A_{\textrm{nonp}}$, the $\mathbb R^d$-boundary of the patch $\mathcal G(T)$  is contained in the union of $\mathcal A_\textrm{nonp}$-tiles of $\mathcal G(T)$ (or rather, their boundaries).
\end{definition}

For the proof of the following result see \cite[Theorem 4.1 and 4.4]{CortezSolomyak:2011}.

\begin{theorem}\label{TheoremTilingRecognizability}
(1) If the dynamical system $(\Omega_{\mathcal G}, \mathbb R^d)$ has no periodic tilings, then the substitution $\mathcal G : \Omega_\mathcal G \rightarrow \Omega_\mathcal G$ is a homeomorphism.

(2) Assume that a substitution $\mathcal G$ satisfies the non-periodic border condition. Then for every tiling $\mathcal T \in \Omega_\mathcal G$ that contains a tile from $\mathcal A_{\textrm{nonp}}$ there exists a unique tiling $\mathcal T'$ such that $\mathcal G(\mathcal T') = \Tk$.
\end{theorem}

\begin{remark}
It is conjectured in \cite[Section 4]{CortezSolomyak:2011} that the NBC condition may be dropped in part (2) of the theorem, that is, $\Gk$ is always invertible on the set of non-periodic tilings.
It is clear that NBC is satisfied in Example~\ref{ex-carpet}. For specific examples, even those which fail the NBC, invertibility of $\Gk$ on non-periodic tilings can sometimes be verified by inspection, by observing that the ``composition,'' discussed after Proposition~\ref{PropositionTilingSurjection}, is unique. In view of Theorem~\ref{TheoremTilingRecognizability}, for tiling systems with NBC, periodic tilings can only exist in minimal components. \cite[Example 4.6]{CortezSolomyak:2011} shows that the latter property may fail without the NBC.
\end{remark}

%%%%%%%%%%%%
%
%
%  Invariant Measures

\subsection{Invariant Measures}

 \begin{definition} (1)  A measure $\mu$ on $\Omega_\mathcal G$   is called {\it invariant}  if  $\mu(U - u) = \mu(U)$ for every $u\in \mathbb R^d$ and every Borel set $U\subset \Omega_\mathcal G$. An invariant measure $\mu$ is called {\it ergodic} if whenever a Borel set $X$ is translation-invariant, i.e.\ $X - u=X$ for every $u\in \mathbb R^d$, either $\mu(X) = 0$ or $\mu(\Omega_\mathcal G \setminus X)=0$.

(2) By the {\it transversal}  of $\Omega_\mathcal G$ we  mean the family of all tilings $\mathcal T \in \Omega_\mathcal G$ such that one of the $\mathcal T$-tiles  is exactly a  prototile from $\mathcal A$. Recall that each prototile contains the origin in the interior of its support. Throughout the paper, the transversal will be denoted by $\Gamma\subset \Omega_\mathcal G$.

(3) A {\it transverse measure} is a Borel measure $\nu$ on $\Gamma$ such that $\nu(U) = \nu(U - u)$ for every Borel subset $U\subset \Gamma$ and $u\in \mathbb R^d$ for which $U - u\subset \Gamma$.
\end{definition}

\begin{proposition} There is a one-to-one correspondence between finite  (resp.\ $\sigma$-finite) transverse measures and finite (resp.\ $\sigma$-finite) invariant measures \cite[Section 7]{CortezSolomyak:2011}.
\end{proposition}

Consider the transversal $\Gamma$. For a prototile $T\in\mathcal A$, set $$\Gamma_T = \{ \mathcal T \in \Gamma : \,T\in \mathcal T \}.$$
Then $\Gamma = \coprod_{T\in \mathcal A}\Gamma_T$ is a disjoint union. The following result provides a description of ``natural'' $\sigma$-finite ergodic measures, for the proof  see Theorems 3.1 and 5.22 in \cite{CortezSolomyak:2011}.

\begin{theorem}\label{TheoremInfiniteInvariantMeasures} (i)  Each finite ergodic measure is supported by one of the minimal components  $\{\Omega_1,\ldots,\Omega_s\}$.

(ii) Let $i \in\{s+1,\ldots,m\}$ be such that the matrix $F_i$ is nonzero and there exist $A\in \mathcal A_i$ and $n>0$ such that a translate of $A$ appears in the interior of $\mathcal G^n(A)$. Then
 there exists a unique (up to scaling)
invariant ergodic $\sigma$-finite measure $\mu$ supported by $\Omega_i$  such that $0< \mu^{\rm tr}(\Gamma_C) < \infty$ for some (and, in fact, for all) prototile $C\in \mathcal A_i$, where
 $\mu^{\rm tr}$ is the transverse measure corresponding to $\mu$. Moreover, the vector $(\mu^{\rm tr}(\Gamma_C))_{C\in \Ak_i}$ is a right Perron-Frobenius eigenvector
of $F_i$.
\end{theorem}

\medskip

\begin{sloppypar}
\begin{remark}
Denote by $\Lk^d$ the Lebesgue measure on $\R^d$.
Observe that the substitution matrix $M_\Gk$ has a strictly positive left eigenvector $(\Lk^d(\supp(T)))_{T\in \Ak}$, corresponding to the eigenvalue $\lam=\rho(A)$.
This follows from (\ref{EquationDefinitionTileSubstitution}) and the fact that the tile boundaries have zero $d$-dimensional Lebesgue measure.
 (The latter is proved
e.g.\ in \cite[Prop.\,1.2]{prag} by B. Praggastis; she does not assume primitivity there.) Note also that $\rho(A_i)=\lam$ for $i\le s$ and
$\rho(A_i) < \lam$ for $i=s+1,\ldots,m$. The latter inequality follows from the existence of strictly positive left eigenvector, see \cite[Theorem III.6]{Gant}.
\end{remark}
\end{sloppypar}

In view of Theorem \ref{TheoremInfiniteInvariantMeasures}, the study of ergodic measures can be reduced to the study of the dynamics on one of the sets $\Omega_i$.
In fact, we can simply restrict the substitution to the subset of prototiles $\bigcup_{j=1}^i\Ak_j$; the substitution matrix will then be obtained by truncating the matrix in (\ref{FrobeniusForm}) so that the diagonal block $F_i$ will be in the lower-right corner. This implies that it is enough to consider substitution tiling systems whose incidence matrices have the following form:

 \begin{equation}\label{MatrixReducedNormalForm}
M_\Gk=\left(
  \begin{array}{cc}
    A & C \\
   0 & B
  \end{array}
\right),
\end{equation}
where $A$ and $B$ are square matrices; $B$ is a primitive matrix; $C$ and $A$ are non-zero matrices. Note that $B$ is uniquely determined (it will be $F_i$ when we consider $\Om_i$); the matrix $A$ { does not have to be primitive} and may contain zero diagonal blocks.
In view of the discussion above, we will always have $\rho(B) < \rho(A)$. Note that $B=[8]$, a $1\times 1$ matrix, in Example~\ref{ex-carpet}.

%%%%%%%%%%%%%%%%%%%%%%%%%%%%%
%
%
%  Technical Assumptions
%

\subsection{Technical Assumptions}\label{SectionAssumptions}

Here  we  summarize the assumptions  we will be implicitly imposing on the  tiling systems in question.  Throughout the paper, the symbols $\mathcal G$ and $\varphi(x) = \lambda\cdot O(x)$, $\Omega_\mathcal G$ will be reserved for a self-similar tile substitution, the associated expansion map, and the tiling space, respectively. The set of prototiles corresponding to the matrix $B$ will be denoted by $\mathcal B$.  Furthermore, we will always assume that the tiling substitution $\mathcal G : \mathcal A\rightarrow \mathcal A^+$ satisfies the following conditions:

\begin{enumerate}

\item Every prototile is a  compact subset of $\mathbb R^d$ that is   the closure of its interior.  Note that  this implies that the Hausdorff dimension of the boundary of every prototile is at least $d-1$, see, for example, Corollary IV.2 and Theorem VII.3 in \cite{hurewicz_wallman}.

\item The tiling system $\Omega_{\mathcal G}$ has finite local complexity.

\item The tile substitution $\mathcal G$ is admissible.

\item The substitution $\mathcal G$ satisfies NBC: the non-periodic border condition (see Definition~\ref{DefNonper}).

\item The substitution matrix $M_\Gk$ has the form (\ref{MatrixReducedNormalForm}), with $C$ non-zero, $B$ primitive, $\rho(B)>1$.

\item We have $\alpha:= \log(\rho(B))/\log(\lambda)> d-1$. The  meaning of $\alpha$ will be clarified in Section \ref{SubsectionGraphDirectedIFS}.
\end{enumerate}

\begin{remark}\label{RemarkTechnAssumptions}
1. The admissibility assumption implies that for every prototile $T\in \mathcal B$ there is $n>0$ such that a translate of $T$ occurs in the interior of $\mathcal G^n(T)$.

2. One can show that if the $\R^d$-boundary of the patches $\Gk(T)$, for $T\in \Bk$, is contained in the union of $\Bk$-tiles of $\Gk(T)$ (this implies NBC), then condition (6) above also holds.
We leave this as an exercise.
\end{remark}

We summarize dynamical properties of tiling systems, which follow from our assumptions.

\begin{proposition}\label{PropositionMainProperties} The space $\Omega_\mathcal G$ is compact. The map $\mathcal G: \Omega_\mathcal G \rightarrow \Omega_\mathcal G$ is a continuous surjection that is invertible on non-periodic tilings. The dynamical system $(\Omega_\mathcal G,\mathbb R^d)$ has a unique (up to scaling) infinite $\sigma$-finite measure $\mu$ such that the corresponding transverse measure is positive and finite on one (equivalently on {\it all}) of the sets $\Gamma_T$, $T\in \mathcal B$.
\end{proposition}

%%%%%%%%%%%%%%%%%%%%%%%%%%%%%
%
%
%  Hierarchical Structure
%

\subsection{ Hierarchical Structure}\label{SectionHierarchical}
For each $k\in \ZZ$, define {\em prototiles of order $k$}  as $\varphi^k(T)=((\varphi^k(\supp(T)), i)$, where $T\in \A$ and $i$ is the label of $T$. Tiles of order $k$ are defined as translates of the prototiles of order $k$. We say that they  have ``type $\Bk$'' if $T$ has type $\Bk$.

 Given a tiling  $\Tk\in \Om_\Gk$, using the surjectivity of the substitution map $\mathcal G$,  find a sequence of tilings $\{\Tk_k\}_{k\in \mathbb Z}$ such that $\Tk_0 = \Tk$ and $\mathcal G (\Tk_k) = \Tk_{k+1}$ for every $k\in \mathbb Z$. Denote by $\Tk^{(k)}$ the tiling obtained from $\Tk_k$ by replacing each tile with the corresponding tile of order $k$, i.e. $\Tk^{(k)} = \varphi^{k}(\Tk_k)$. Note that the tiles of $\Tk^{(k+1)}$ are obtained from tiles of $\Tk^{(k)}$ by ``composition,'' roughly speaking, by taking appropriate unions, and the inverse operation is ``subdivision,'' determined by the substitution rule. Observe that if $\Tk$ contains a tile of type $\Bk$, then these tilings are uniquely defined; in fact,
  \beq \label{EquationHierar}
\Tk^{(k)} = \varphi^k \Gk^{-k}(\Tk)\ \ \mbox{for all}\ k\in \ZZ.
\eeq

 The tiles of $\Tk^{(k)}$ will be called {\em tiles of order $k$ obtained from $\Tk$}.
The tiles of $\Tk^{(k)}$ for $k>0$ will sometimes be referred to as ``supertiles of $\Tk$''.
%The main use of this definition comes from the observation that if $T_k$, $k>0$, is a supertile of order $k$ containing the origin, then replacing the support of $T_k$ by the  set  $\varphi^{-k}(\mathcal G^{k}(T_k))$, we will get a ``large'' portion of $\Tk$ covering the origin. Thus, in many cases, the tiling $\Tk$ can be uniquely restored from its supertiles.

%%%%%%%%%%%%%%%%%%%%%%%%%%%%%%%%%%%%%%%%%%%%%%%%%%%%%%%%%%%
%
%  Transverse Dynamics
%

\section{Transverse Dynamics}\label{SectionTransverseDynamics}

One of the main technical ingredients in our proof will be the dynamical system $(\Omega_{\mathcal G}, \mathcal G)$, restricted to a certain fractal subset defined below. Recall that  $\varphi = \lambda \cdot O$, $\lambda >1$, and $O$ is an orthogonal matrix.

\subsection{Graph-Directed Iterated Function Systems}\label{SubsectionGraphDirectedIFS}

Consider the directed  graph $G = (V,E)$ such that the set of vertices $V$ coincides with the alphabet $\mathcal B$ and the multiplicity of the set of edges from $T$ to $T'$ is exactly the number of occurrences of the (translate of) prototile  $T'$ in the patch $\mathcal G(T)$. It follows that the transpose  matrix  $B^t$ is exactly the incidence matrix of the graph $G$. We will denote by $\mathcal E_{T,T'}$ the set of edges connecting a vertex $T$ to a vertex $T'$.
We will use the symbols $s(e), r(e)$ respectively, to denote the source and range of a directed edge.

For each vertex (prototile) $T\in V$, consider the set $S_T =\supp(T)$.
  Notice that  there is a one-to-one correspondence between edges in $\mathcal E_{T_1,T_2}$ and the set of translates of (the order $(-1)$ tile) $\varphi^{-1}(T_2)$ in $\varphi^{-1}(\mathcal G(T_1))$.  {\it We shall fix such a correspondence}.

This can be made precise using formula (\ref{EquSub}) for the tile substitution: the set of edges $\Ek_{T,T'}$ corresponds to $\Dk_{T,T'}$, for $T,T'\in \Bk$.
Given $e\in \Ek_{T,T'}$, we denote the corresponding vector in $\Dk_{T,T'}$ by $u_e$. Then the similitude
\beq\label{EquSim}
f_e:\,x\mapsto \varphi^{-1}(x+u_e)
\eeq
maps the set  $S_{T'}$ onto the translate of $\varphi^{-1}(S_{T'})$ corresponding to the edge $e$ in the patch $\varphi^{-1}(\mathcal G(T))$ according to (\ref{EquSub2}). Observe that distinct edges define different maps.  Then $G = (V,E)$, $\{f_e\}_{e\in E}$, is a  {\it graph-directed system}, and it
uniquely defines  a  family of non-empty compact sets $\{K_T\}_{T\in \mathcal B}$ of $\mathbb R^d$ such that
\begin{equation}\label{EquationUnionGraphDirectedSets}K_T = \bigcup_{T'\in \mathcal B} \bigcup_{e\in \mathcal E_{T,T'}}f_e(K_{T'}),\end{equation} see \cite{MauldinWilliams:1988} or \cite[p.48]{Falconer:Techniques}. Note that $K_T\subset S_T$  for every $T\in \mathcal B$. The set $K_T$ is obtained from $S_T$ by consecutively removing all ``$\varphi$-preimages'' of tiles from $\mathcal A\setminus \mathcal B$. We note that the union (\ref{EquationUnionGraphDirectedSets}) need not  be disjoint.

Observe that the contraction coefficient of every map $f_e$, $e\in E$, is exactly $1/\lambda$. Thus, to find the Hausdorff dimension of the sets $\{K_T\}_{T\in \mathcal B}$, one needs to consider the matrix $D^{(s)}$ with the entries  $$D_{T_1,T_2}^{(s)} = \sum_{e\in \mathcal E_{T_1,T_2}}\frac{1}{\lambda^s} = \frac{1}{\lambda^s}m_{T_2,T_1}. $$
It follows from \cite{MauldinWilliams:1988}, see also   \cite[Corollary 3.5]{Falconer:Techniques}\footnote{This requires the {\em open set condition} which  can be verified by setting $U_T = \textrm{int}(S_T)$ and noting that $U_T \supset \bigcup_{T'\in \mathcal B} \bigcup_{e\in \mathcal E_{T,T'}}f_e(U_{T'})$, the union being disjoint.}, that the Hausdorff dimension of each set $K_T$, $T\in \mathcal B$, is the unique positive number $\alpha$ satisfying $$1 = \rho(D^{(\alpha)}) = \frac{1}{\lambda^\alpha}\rho(B^t) = \frac{\rho(B)}{\lambda^\alpha}.$$  Therefore,  the Hausdorff dimension of every set $K_T$, $T\in \mathcal B$, is equal to \begin{equation}\alpha = \log(\rho(B))/\log(\lambda).\end{equation}

\medskip

\begin{remark}\label{RemarkPosmeas}
It is proved in \cite{MauldinWilliams:1988}, see also \cite[Corollary 3.5]{Falconer:Techniques}, that the $\alpha$-dimensional Hausdorff measure of $K_T$, denoted $\Hk^\alpha(K_T)$, is positive and finite. This, together with (\ref{EquationUnionGraphDirectedSets}), implies that
\beq \label{EqDis}
\Hk^\alpha(K_T\cap K_{T'})=0\ \ \mbox{for}\ T\ne T',
\eeq
and $(\Hk^\alpha(K_T))_{T\in \Bk}$ is a left Perron-Frobenius eigenvector of $B$.
\end{remark}

\begin{sloppypar}
We will also need the fact that
$\{\Hk^\alpha|_{K_T}\}_{T\in \Bk}$  is the list of {\em natural self-similar graph-directed measures} on the attractors. This means that
$\{\Hk^\alpha|_{K_T}\}_{T\in \Bk}$  is, up to scaling, the unique list of finite and positive Borel  measures $\eta_{_T}$ on $K_T$, $T\in \Bk$, such that
\beq \label{EqNat}
\eta_{_T} = \sum_{T'\in \Bk} \sum_{e\in \Ek_{T,T'}}\frac{1}{\rho(B)} (\eta_{_{T'}}\circ f_e^{-1}),
\eeq
see \cite[3.5]{Edgar:1997}. Moreover, these natural measures may be obtained as projections of
appropriate Markov measures on the sequence space, as we are now going to explain.
\end{sloppypar}

Let $T\in \Bk$.
By the definition of graph-directed sets, $x\in K_T$ if and only if
there is an infinite path $(e_0,e_1,\ldots)$ in the graph $G$ such that $T=s(e_0)$ and
\begin{equation}\label{EquationControlPoint0}
\{x\} = \bigcap\limits_{k = 0}^\infty   K_{(e_0,\ldots,e_k)},
 \end{equation}
 where
  \begin{equation} \label{EquationIter}
  K_{(e_0,\ldots,e_k)} = f_{e_0}\circ f_{e_1}\circ \cdots \circ f_{e_k}(K_{r(e_k)}).
  \end{equation}

\begin{definition}\label{DefShift}
(1) Let $X_G$ be the two-sided edge shift space associated with the graph $G = (V,E)$, i.e.
 $$X_G = \{(e_n) \in  E^\mathbb Z : e_{n+1}\mbox{ follows }e_n\mbox{ in the graph }G\mbox{ for every }n\in \mathbb Z\}.$$
Formally, ``$e_{n+1}\mbox{ follows }e_n\mbox{ in the graph }G$'' means $r(e_{n+1})=s(e_n)$.
The left shift on $X_G$ is denoted by $S$.

(2) We will  refer to $X_G$ as the set of
infinite two-sided paths $(e_n)_{n\in\Z}$ in the graph $G$. We will need it in the next section; for now, let  $X_G^+$ be the set of one-sided infinite paths $(e_n)_{n\ge 0}$ in the graph $G$. The {\em natural projection} $\pi_+:\,X_G^+\to \R^d$ is defined by
$$
\pi_+\left((e_n)_{n\ge 0}\right) = \lim_{k\to \infty}  f_{e_0}\circ f_{e_1}\circ \cdots \circ f_{e_k}(x_0),
$$
which is independent of $x_0\in \R^d$.
 For $T\in \Bk$ let
\beq \label{EquDur}
X^+_G(T) = \{(e_n)_{n\ge 0}\in X_G^+:\ s(e_0) = T\}
\eeq
be the set of infinite paths in $G$ starting at the vertex $T$. Clearly, $X_G^+ = \coprod_{T\in \B} X^{+}_G(T) $ is a disjoint union.
It follows from (\ref{EquationControlPoint0}) that
$$
K_T = \pi_+(X^+_G(T) ).
$$
Using (\ref{EquSim}), the natural projection $\pi_+$ can be written explicitly as follows:
\beq \label{EquNatproj}
\pi_+\left((e_n)_{n\ge 0}\right) = \sum_{n=0}^\infty \varphi^{-n-1} u_{e_n}.
\eeq

(3) Let $w = (w_T)_{T\in \Bk}$ be the right Perron-Frobenius eigenvector for the matrix $B^t$, such that $\sum_{T\in \Bk} w_T =1$.
We consider the Markov measure $\ov{\eta}$ on $X_G^+$, with initial probabilities (i.e.\ probabilities of starting at $T\in \Bk$) equal to $w_T$ and the probability of moving along an edge
$e$ equal to $\frac{w_{r(e)}}{\rho(B) w_{s(e)}}$. Consistency follows from the fact that $w$ is the right eigenvector of the transition matrix for the graph, with eigenvalue $\rho(B)$.
 In view of
Remark~\ref{RemarkPosmeas}, we have
$$
w_T = c_0^{-1} \Hk^\alpha(T)\ \ \mbox{where}\ \ c_0=\sum_{T'\in \B}\Hk^\alpha(T'),
$$
and hence for a cylinder set $[e_0,\ldots,e_n]\subset X_G^+$ we obtain
\beq \label{EqCylmeas}
\ov{\eta}([e_0,\ldots,e_n])  = \frac{w_{r(e_n)}}{\rho(B)^{n+1}}=\frac{\Hk^\alpha(r(e_n))}{c_0\rho(B)^{n+1}},.
\eeq
\end{definition}

The next lemma follows from the theory of graph-directed IFS (see e.g.\ the proof of \cite[Theorem 3]{MauldinWilliams:1988}).

\begin{lemma} \label{LemNat}
For $T\in \Bk$ consider
$$
\eta_{_T}:= \ov{\eta}|_{X^{+}_G(T) }\circ \pi_+^{-1},
$$
that is, the natural projection of the measure $\ov{\eta}$ restricted to $X^{+}_G(T) $. Then $\{\eta_{_T}\}_{T\in \Bk}$ is the list of graph-directed self-similar measures satisfying (\ref{EqNat}), and
$$
\eta_{_T} = c_0^{-1}\Hk^\alpha|_{K_T}\ \ \mbox{for}\ T\in \Bk,
$$
where $c_0=\sum_{T'\in \B}\Hk^\alpha(T')$.
\end{lemma}

%%%%%%%%%%%%%%%%%%%%%%%%%%%%%%%%%%%%%%
%
%  The Tile Substitution
%

%$$\dot{\bigcup}, \amalg, \coprod$$

\subsection{``Cantorization'' of tilings}

  Recall that the graph-directed set $K_T$ is defined for every prototile $T\in \mathcal B$. If $T' = T +x$ is a translate of a prototile $T\in \mathcal B$, then we  write $K_{T'}$ for the set $K_{T} + x$.
 %\end{definition}

\begin{definition} (1) The {\it ``cantorization"} of a tiling $\mathcal T \in\Omega_\mathcal G$ is the set
$$
\C(\T) = \bigcup\{K_T:\ T\in \T\ \mbox{and $T$ is type $\B$}\}.
$$

(2) Denote $$\Omega_0 = \{\mathcal T\in \Omega_\mathcal G : \textbf{0}\in \mathcal C(\mathcal T)\},$$ where $\textbf{0}$ stands for the zero vector.
\end{definition}

Next we present some equivalent conditions for the property $\Tk\in \Omega_0$, which are immediate from the definitions.

\begin{remark}\label{RemarkNested}
(1) We have $\Tk\in \Omega_0$ if and only if there exist $T_0\in \Bk$ and $x\in K_{T_0}$ such that $T_0-x\in \Tk$.

(2) We have $\Tk\in \Omega_0$ if and only if there is a nested sequence of type $\B$ tiles of order $-k$ obtained from $\Tk$, such that the intersection of their supports is the origin $\b0$. More formally (compare (\ref{EquationHierar})), we have that $\Tk\in\Omega_0$ if and only if there is a sequence of type $\Bk$ tiles $T_{-k}\in \varphi^{-k}\Gk^k(\Tk)$, $k\ge 1$, such that
$$
\supp(T_{-k-1})\subset \supp(T_{-k}),\ k\ge 0,\ \ \mbox{and}\ \ \bigcap_{k=0}^\infty \supp(T_{-k})=\{\b0\}.
$$
\end{remark}

\begin{proposition}\label{PropositionPropertiesGeometricSubstitution} (i) The set $\Omega_0$ is compact in the tiling metric.

(ii) The map $\mathcal G : \Omega_0\rightarrow \Omega_0$ is a homeomorphism.

(iii) For every tiling $\mathcal T \in \Omega_0$, we have that $\mathcal C(\mathcal G^{-1}(\mathcal T)) = \varphi^{-1}(\mathcal C(\mathcal T))$ and $\mathcal C(\mathcal G(\mathcal T)) = \varphi(\mathcal C(\mathcal T))$.
\end{proposition}
\begin{proof} (i) We only need to show that the set $\Omega_0$ is closed in $\Omega_\mathcal G$. Consider a tiling $\mathcal T \notin \Omega_0$. Then $\textbf{0}\notin \mathcal C(\Tk)$.
Then $\b0$ does not belong to $K_T$ for any tile of type $\B$ containing the origin, which is an open condition, since $K_T$ is compact.

(ii) The continuity of the map $\mathcal G :\Omega_\Gk\rightarrow \Omega_\Gk$ is well-known and easily follows from the definition of the tiling metric.
 Note that tiles of type $\B$ belong to $\Anonp$, since they do not appear in minimal components, see Definition~\ref{DefNonper}(1). So $\Gk$ is one-to-one on $\Om_0$. Hence, in view of Theorem \ref{TheoremTilingRecognizability}, we only need to show that $\mathcal G(\Omega_0) = \Omega_0$.

Remark~\ref{RemarkNested}(2) implies that if $\Tk\in \Omega_0$ then $\Gk(\Tk)\in \Omega_0$. Indeed, using the notation of the remark,
we have $\varphi T_{-k} \in \varphi^{-(k-1)}\Gk^{k-1}(\Gk(\Tk))$, so $\{\varphi T_{-k}\}_{k\ge 1}$ is a nested sequence of tiles of order $-(k-1)$ obtained from $\Gk(\Tk)$, all of
type $\Bk$, and clearly the intersection of their supports is $\{\b0\}$.

%Equations (\ref{EquationControlPoint0}) and (\ref{EquationControlPoint0}) imply that if $\mathcal T\in \Omega_0$, then $\mathcal G(\mathcal T)\in \Omega_0$.

Since the map $\mathcal G$ is invertible on $\Omega_0$ by Theorem \ref{TheoremTilingRecognizability}, we can find a (unique) tiling $\mathcal T_{-1}$ with $\mathcal G(\mathcal T_{-1}) = \mathcal T$. We claim that $\Tk_{-1}\in \Omega_0$. If not, then for some $k>0$ all the tiles of order $-k$ obtained from $\Tk_{-1}$ containing the origin are of type $\A\setminus \B$. But the substitution of $\A\setminus\B$ tiles contains only $\A\setminus\B $ tiles,
so we get a contradiction with the assumption that the origin lies in a $\B$ tile of $\Tk$.
We have proved that $\mathcal G|_{\Omega_0}$ is a homeomorphism.

(iii)   It follows from Equation (\ref{EquationUnionGraphDirectedSets}) that $\varphi^{-1}(\C(\Gk(T)))=K_T$ for every prototile $T\in \Bk$. Therefore,
$$
\varphi^{-1}(\C(\Gk(\T)))=\C(\T)\ \ \mbox{for every tiling}\ \T\in \Omega_0.
$$
Ths implies the last statement of the proposition.
\end{proof}

%%%%%%%%%%%%%%%%%%
%
%  Transverse Measures

\section{Transverse Measures}\label{SectionTransverseMeasures}

In this section we show that the dynamical system $(\Omega_0,\mathcal G)$ is measure-theoretically isomorphic to a mixing Markov chain, namely, the edge
shift on the graph $G$ with the incidence matrix $B^t$, equipped with the measure of maximal entropy. We note that properties of $(\Omega_\G,\mathcal G)$ as a topological dynamical system were earlier considered in \cite{Anderson_Putnam:1998}.

Let $\mu$ be the  $\mathbb R^d$-ergodic measure on $\Omega_\Gk$ as described in Proposition \ref{PropositionMainProperties}. It is unique up to scaling; we will normalize it later.
%Choose $\eta>0$ so that every prototile $A\in \mathcal B$ contains the ball $B_\eta(\textbf{0})$ in its interior.
There exists a unique Borel $\sigma$-finite transverse measure $\mutr$ on the transversal $\Gamma$ such that
\beq \label{EqProduct0}
\mu(U - \Theta) = \mutr(U)\cdot \Lk^d(\Theta),
\eeq
where $\mathcal L^d$ is the $d$-dimensional Lebesgue measure and
$$
U-\Theta = \{\T-x:\ \T\in U,\ x\in \Theta\},
$$
for all Borel sets $ U\subset \Gamma_Q$ and $\Theta\subset \supp(Q),\ Q\in \Bk$,
see Section 7 in \cite{CortezSolomyak:2011} for the details. (Actually, in \cite{CortezSolomyak:2011} this is only proved for $U$ contained in a small ball centered at the origin, but the formula in stated generality follows from shift invariance of the measure $\mu$.)
This means that ``locally'' the measure $\mu$ behaves as a product measure.

Following \cite{CortezSolomyak:2011}, we give the following definition. Recall that $\Gamma_{T} = \{\mathcal T\in \Gamma :\, T\in \mathcal T\}$.

\begin{definition} \label{DefTrans} For every $Q\in \mathcal B$ and $n\geq 0$, define $$\mutr_{n,Q} = \mutr(\mathcal G^n(\Gamma_{Q}) - x),$$ where $x$ is a vector such that $\mathcal G^n(\Gamma_{Q}) - x\subset \Gamma$. Since $\mutr$ is a transverse measure, the definition of $\mutr_{n,Q} $ does not depend on the choice of $x$.
\end{definition}

The next result follows from Lemma 5.11 in \cite{CortezSolomyak:2011} and the Perron-Frobenius theorem for  primitive matrices.
% Recall that  $B$ stands for the matrix  corresponding to the restriction of $\mathcal G$ on the alphabet $\mathcal B$. We  use the symbol $\rho(B)$ to denote the %spectral radius of $B$.

\begin{lemma}\label{LemmaTransverseMeasureEignevectors}  There exists a (unique) right Perron-Frobenius eigenvector  $\xi$  for the matrix $B$ such that
\beq \label{PF}
\mutr_{n,{Q}} = \frac{\xi_Q}{\rho(B)^n}\mbox{ for every }Q\in \mathcal B\mbox{ and }n\geq 0.
\eeq
\end{lemma}

Let  $G=(V,E)$ be the graph of the iterated function system constructed in Section \ref{SubsectionGraphDirectedIFS}.
%For each edge $e\in E$, denote by $s(e)$ and $r(e)$ its source and range, respectively.

\begin{definition}\label{DefCorresp}
(1) The {\em itinerary} of $\Tk\in \Om_0$ for the $\Gk$-dynamics is a two-sided infinite path $\beta(\Tk)=(e_n)_{n\in \Z}\in X_G$ (see
Definition~\ref{DefShift}), defined as follows:
$\beta(\Tk) = (e_n)_{n\in \Z}$ if for all $n\in \Z$ the tiling
$\Gk^n(\Tk)$ has a tile $T_n$ of type $s(e_n)=r(e_{n-1})\in \Bk$
containing the origin and $T_{n+1}$ occurs in $\Gk(T_n)$ in the position corresponding to $e_n$.

(2) Observe that $\varphi^n T_{-n}$ for $n> 0$ forms an increasing sequence of supertiles of the tiling $\Tk$. We will call it the {\em compatible sequence of
supertiles of $\Tk$ containing the origin}.

(3) Note that the itinerary need not be unique. Denote by $\Om_0^*$
the set of all tilings $\mathcal T\in \Omega_0$ for which the itinerary is unique. The itinerary is non-unique if and only if for some $n\in \Z$, the origin $\b0$ lies on the common boundary of two tiles of $\B$ type $T_n, T_n'\in \G^n(\T)$ and, moreover, $\b0\in K_{T_n}\cap K_{T_n'}$.
Note that just being on the boundary of a tile may not lead to non-uniqueness.

Thus $\beta:\,\Om_0^*\to X_G$ is a well-defined function, whereas $\beta$ may be considered as a multi-valued function on all of $\Om_0$.
Observe that $\Om_0^*$ is a $\Gk$-invariant Borel subset of $\Om_0$.

(4)  By definition, $\beta\circ \Gk = S\circ \beta$, where $S$ is the left shift on $X_G$. This holds, in an appropriate sense, even when the itinerary is
non-unique.
\end{definition}

\begin{remark}\label{RemarkCorresp}
(1) Many properties of the tiling dynamical system can be expressed using the symbolic dynamics provided by the itinerary $\beta$. In particular, if we fix the
left one-sided sequence $(e_n)_{n< 0}$, this corresponds to the set of translates of $\Tk\in \Omega_0^*$, such that the origin stays in the interior of its original tile
of type $r(e_{-1})$. This is a ``piece'' of the translation orbit of $\Tk$. On the other hand, fixing the right half of the symbolic orbit
$(e_n)_{n\ge 0}$ corresponds to the transversal; more precisely, for all $\Tk \in (\Gamma_T+x)\cap\Omega^*_0$ for a fixed vector $x$,
the sequences $\beta(\Tk)$ agree in $n\ge 0$.

(2) There are, however, some complications. First, $\beta$ is not well-defined on $\Omega_0\setminus\Omega_0^*$.
Second, $\beta$ need not be one-to-one and need not be onto (even if extended to $\Omega_0$ as a multi-valued function). The reason is that the sequence $(e_n)_{n< 0}$ determines a sequence of compatible supertiles {\it whose union need not be the entire space $\R^d$}. We will deal with such sequences by
showing that they have zero measure of maximal entropy for $S$. To get an example of such a sequence, let $\Gk$ be the substitution from Example~\ref{ex-carpet}. Then the graph $G$ has a single vertex and eight loops, corresponding to the eight 1-tiles in the substitution of a 1-tile. Taking a constant sequence $(e_n)_{n< 0}$, corresponding to the lower-left 1-tile, for instance, will yield the tiling of the 1st quadrant of the plane; see the discussion following Definition~\ref{def-FLC}. More generally, if the sequence $(e_n)_{n< 0}$ eventually consists of edges corresponding to the 1-tiles on one of the sides of the substituted 1-tile, then the union of compatible supertiles will only cover a half-plane or a quarter-plane.
\end{remark}

\begin{definition}
(1) Define $X^*_G$ as the set of $(e_n)\in X_G$ such that the compatible increasing sequence of supertiles, corresponding to $(e_n)_{n< 0}$, has all of $\R^d$ as the limit (i.e.\ the union) of the supports. It is clear that $X^*_G$ is $S$-invariant.

(2) We define the {\em natural projection map} $\pi:\,X^*_G\to \Omega_0$ so that $(e_n)$ is an itinerary of $\Tk:=\pi(\ov{e})$. It is possible to describe $\Tk$ explicitly, as a limit of an increasing compatible sequence of patches (whose supports are the supports of supertiles of $\Tk$).
The condition $\ov{e}=(e_n)_{n\in\Z}\in X_G$ means, by definition, that
\beq \label{Nest}
r(e_n) + u_{e_n} \in \Gk(s(e_n))\ \ \ \mbox{for all}\ n\in \Z
\eeq
(recall that the vertices of $G$ are identified with the prototiles in $\B$).
A tile of $\Tk =\pi(\ov{e})$ containing the origin (possibly non-unique) must be
\beq \label{Proj+}
T_0 - \sum_{n=0}^\infty \varphi^{-n-1}u_{e_n}=T_0 - \pi_+(\ov{e}_+),
\eeq
where $T_0 = s(e_0)$ and $\ov{e}_+ = (e_n)_{n\ge 0}$ (recall that $\pi_+$ was defined in Definition~\ref{DefShift}(2)). Note that this already guarantees $\T\in \Om_0$, in view of Remark~\ref{RemarkNested}(1) and
(\ref{EquNatproj}). Now we let
\beq \label{EquProj1}
\pi(\ov{e}) = \lim_{k\to \infty} \left[\Gk^k(s(e_{-k})) - \sum_{n=-k}^\infty \varphi^{-n-1} u_{e_n}\right].
\eeq
We claim that these patches are increasing and compatible. Indeed,
$$
\Gk^k(s(e_{-k})) - \sum_{n=-k}^\infty \varphi^{-n-1} u_{e_n}\subset \Gk^{k+1}(s(e_{-k-1})) - \sum_{n=-k-1}^\infty \varphi^{-n-1} u_{e_n}
$$
reduces to
$$
u_{e_{-k-1}} + s(e_{-k}) \in \Gk(s(e_{-k-1})),
$$
which follows from (\ref{Nest}), keeping in mind that $s(e_{-k}) = r(e_{-k-1})$. Thus, the right-hand side of (\ref{EquProj1}) is well-defined, and it is a tiling
of the entire $\R^d$ if $\ov{e}\in X_G^*$.
\end{definition}

\begin{lemma} \label{LemmaCorresp}
We have $\pi\circ S= \Gk\circ\pi$ on $X^*_G$ and $\pi\circ\beta ={\textit Id}$ on $\beta(\Omega_0^*)$.
\end{lemma}

\begin{proof}
This is an immediate consequence of the definitions.
\end{proof}

Next we consider the measure of maximal entropy (the Parry measure) $\ov{\nu}$ for the shift $S$ on $X_G$. Recall that the incidence matrix for the graph $G$ is $B^t$. The Parry measure (of the edge shift) is a Markov measure, given by
\beq \label{Parry0}
\ov{\nu}([e_1,\ldots,e_{n}]_k) = u_{s(e_1)} v_{s(e_1)} \prod_{j=1}^{n} \frac{v_{r(e_{j})}}{\rho(B) v_{s(e_{j})}}=\frac{u_{s(e_1)} v_{r(e_{n})}}{\rho(B)^{n}},
\eeq
where $[e_1,\ldots,e_{n}]_k$ is a cylinder set in $X_G$ starting at the index $k \in \Z$,  $n\ge 0$, $u=(u_Q)_{Q\in \Bk}$ is the left Perron-Frobenius eigenvector of
$B^t$, and $v=(v_Q)_{Q\in \Bk}$ is the right Perron-Frobenius eigenvector of $B^t$, normalized so that $\sum_{Q\in \Bk} u_Q v_Q=1$. The measure is clearly
shift-invariant.
We have
$$
\ov{\nu}(\{\ov{e}\in X_G:\ s(e_0) = Q\}) = \sum_{Q'\in \Bk} \sum_{e\in \Ek_{Q,Q'}} u_Q v_{Q'}\cdot \rho(B)^{-1} = u_Q v_Q,
$$
which implies that $\ov{\nu}$ is a probability measure. For the vector $u$
we can take the vector $\xi$ from Lemma~\ref{LemmaTransverseMeasureEignevectors}, which is a right Perron-Frobenius eigenvector for $B$, and for the vector $v$ we can
take $(\Hk^\alpha(K_Q))_{Q\in \Bk}$, which is a left Perron-Frobenius eigenvector for $B$, see Remark~\ref{RemarkPosmeas}. Since the measure $\mu$ was defined up to scaling, we can normalize it (this will also affect the transverse measure)
so that
\beq \label{EqNormal1}
\sum_{Q\in \Bk} \xi_Q \Hk^\alpha(K_Q)=\sum_{Q\in \Bk} \mutr(\Gamma_Q) \Hk^\alpha(K_Q)=1.
\eeq
Then we have
\beq \label{Parry}
\ov{\nu}([e_1,\ldots,e_{n}]_k) =\frac{\xi_{s(e_1)}}{\rho(B)^{n}} \,\Hk^\alpha(K_{r(e_{n})} ) \ \ \mbox{for}\ k\in \Z,\ n\ge 0.
\eeq
It is well-known that the measure-preserving transformation $(X_G,S,\ov{\nu})$ is ergodic, where $S$ is the left shift.

\begin{lemma}\label{LemmaExcep1}
We have $\ov{\nu}(X_G\setminus X^*_G)=0$.
\end{lemma}

\begin{proof}
Recall that there is an integer $k>0$ such that for each prototile $T\in \Bk$ the interior of the patch $\Gk^k(T)$ contains a translate of $T$. We can assume
without loss of generality, passing from $\Gk$ to $\Gk^k$, that $k=1$. Then for any vertex of the graph $G$ (i.e.\ $Q\in \Bk$) there is an edge $e$, with
$s(e)=Q$, which corresponds to the choice of an interior tile in the patch $\Gk(Q)$. A one-sided path $(e_n)_{n<0}$, which includes infinitely many of these ``interior''
edges, will necessarily belong to $X_G^*$. Indeed, choosing an interior supertile of order $n$ inside the supertile of order $n+1$, for $n>0$, implies that the union of the
compatible sequence of supertiles contains the ball of radius $\delta\lam^n$ centered at the origin, for some $\delta>0$.
A standard argument shows that this is a full measure set. To verify this, note that the set of paths, which avoid the selected edges, has a growth rate equal to
the spectral radius of a matrix $B'$ having at least one entry in each row smaller than that of $B$, whence $\rho(B')<\rho(B)$.
\end{proof}

\begin{definition}\label{DefNu}
Define the measure $\nu$ on $\Omega_0$ as the ``push-forward'' of $\ov{\nu}$ on $X_G^*$ via the map $\pi$, that is,
$$
\nu(U) = \ov{\nu}(\pi^{-1}(U))\ \ \ \mbox{for Borel} \ U\subset \Omega_0.
$$
\end{definition}

Since $\ov{\nu}$ is $S$-invariant, we obtain from
Lemma~\ref{LemmaCorresp} that the measure $\nu$ is $\Gk$-invariant on $\Om_0$.

\begin{lemma} \label{LemmaExcep2}
We have $\nu(\Om_0\setminus\Om_0^*)=0$.
\end{lemma}

\begin{proof}
The argument is almost the same as in the proof of Lemma~\ref{LemmaExcep1}. It is enough to prove that the set of tilings $\T\in \Om_0$, for which there exists
$n\in \Z$ such that $\b0$ is on the boundary of a tile of type $\B$ in $\G^n(\T)$, has $\nu$ measure zero. Considering the itineraries of such tilings, we see that
they must contain only finitely many edges $e_i$ corresponding to the tile of type $r(e_i)$ in the interior of $\G(s(e_i))$, for $i\ge 0$. But the growh rate of
such sequences is strictly less than $\rho(B)$, hence their $\ov{\nu}$ measure equals zero, as desired.
\end{proof}

\begin{theorem} \label{ThConj} Suppose that the Markov measure $\ov{\nu}$ is defined by (\ref{Parry}), using the normalization (\ref{EqNormal1}), and
$\nu = \ov{\nu}\circ \pi^{-1}$. Then the following hold:

(i) The probability-preserving system $(\Om_0,\G,\nu)$ is measure-theoretically isomorphic to $(X_G,S,\ov{\nu})$, hence ergodic.

(ii) For any $Q\in \Bk$ and all  Borel sets $\Theta \subset \Gam_Q,\ W\subset K_Q$ we have
\beq\label{EqProd}
\nu(\Theta-W) = \mutr(\Theta)\cdot \Hk^\alpha(W).
\eeq
\end{theorem}

\begin{proof} (i) This follows from Lemmas~\ref{LemmaCorresp}, \ref{LemmaExcep1}, and \ref{LemmaExcep2}.

(ii) The left-hand side of (\ref{EqProd}) is well-defined, since $\Gam_Q-K_Q\subset \Om_0$ by Remark~\ref{RemarkNested}(1). First let us prove the equality for $\Theta =
\Gam_Q$. Recall that $X_G^+(Q)$ denotes the set of one-sided paths in $G$ starting at $Q$.
 We have
$$
\nu(\Gam_Q-W) = \ov{\nu}\left(\{\ov{e}\in X_G:\ \ov{e}_+\in X^+_G(Q)\ \mbox{and}\ \pi_+(\ov{e}_+)\in W\}\right),
$$
using the fact that $\Tk\in \Gam_Q-W$, with $W\subset K_Q$, has an itinerary with $s(e_0)=Q$, and $\nu$ almost every tiling has a unique itinerary by
Lemma~\ref{LemmaExcep2}. The measure $\ov{\nu}$ on $X_G$ induces a measure $\ov{\nu}_+$ on $X_G^+$ via the projection $\ov{e}\mapsto \ov{e}_+$.
 Comparing (\ref{Parry0}) with (\ref{EqCylmeas}) we see that
$$
\ov{\nu}_+|_{X^+_G(Q)} = c_0^{-1} \xi_Q\cdot\ov{\eta}|_{X^+_G(Q)}.
$$
Thus,
\begin{eqnarray*}
\nu(\Gam_Q-W) & = & \ov{\nu}_+|_{X^+_G(Q)}(\pi_+^{-1} W) \\
                           & = & c_0^{-1} \xi_Q\cdot\ov{\eta}|_{X^+_G(Q)}(\pi_+^{-1} W) \\
                           & = &  c_0^{-1} \xi_Q\cdot\eta_Q(W) \\
                           & = &  \xi_Q \Hk^\alpha(W),
\end{eqnarray*}
where we used Lemma~\ref{LemNat} in the last step.

Now let us verify (\ref{EqProd}) for an arbitrary Borel $\Theta\subset \Gam_Q$. The transversal $\Gam_Q$ is topologically a Cantor set, in which the Borel
$\sigma$-algebra is generated by the sets of the form $\G^n(\Gam_{Q'}-x)$, $Q'\in \A$, where $x$ is such that $Q+x\in \G^n(Q')$. We have
$$
\nu(\G^n(\Gam_{Q'} - x-W) = \nu(\G^n(\Gam_{Q'} - \varphi^{-n}(x+W))) = \nu(\Gam_{Q'} -\varphi^{-n}(x+W)),
$$
using the fact that $\nu$ is $\G$-invariant. Note that $W\subset K_{Q}$ and $Q+x\in \G^n(Q')$ imply $W+x\subset \C(\G^n(Q'))=\varphi^n K_{Q'}$, hence
$\varphi^{-n}(x+W)\subset K_{Q'}$, and by the case of (\ref{EqProd}) already proved,
\begin{eqnarray*}
\nu(\Gam_{Q'} -\varphi^{-n}(x+W))&=&\xi_{Q'} \cdot \Hk^\alpha(\varphi^{-n}W) \\
                                                     &=& \frac{\xi_{Q'}\cdot \Hk^\alpha(W)}{\lam^{nd}} \\
                                                     & = & \frac{\xi_{Q'}\cdot \Hk^\alpha(W)}{\rho(B)^{n+1}} \\
                                                     & = & \mutr(\Gam_{Q'} - x)\cdot\Hk^\alpha(W),
\end{eqnarray*}
by Lemma~\ref{LemmaTransverseMeasureEignevectors} and Definition~\ref{DefTrans}. The proof is complete.

\end{proof}

%%%%%%%%%%%%%%%%%%%%%%%%%%%%%%%%%%%%%%%%%%%%%%%%%%%%%%%%%%%%%
%  Ergodic theorem
%
%%%%%%%%%%%%%%%%%%%%%%%%%%%%%%%%%%%%%%%%%%%%%%%%

\section{Second Order Ergodic Theorem}\label{SectionSecondOrderErgodicTheorem}

In this section we establish the second order theorem for tiling substitution systems.  We  begin with Lemma \ref{LemmaFisherSecondOrder} saying that  the second order ergodic theorem can be established by checking the convergence of second order averages for one (any) function only. This lemma  was originally proved in \cite[Theorem 4]{Fisher:1992} for the discrete case. We include the proof for the reader's convenience. The proof is based on the following generalization of Hopf's ratio ergodic theorem. Recall that a group action is {\em free} if the identity is the only group element for which there exists a fixed point. Our tiling translation
action is free in the measure-theoretic sense, since tilings containing at least one tile of type $\B$ are non-periodic \cite[Corollary 4.5]{CortezSolomyak:2011}, and these tilings form an invariant set of full
$\mu$ measure. Recall that $B_R$ denotes the closed Euclidean ball. We note that the dynamical system in the following theorem need not be conservative as the ergodic sums get averaged over symmetric balls versus $[0,n]^d$.

\begin{theorem}[M.~Hochman \cite{Hochman:2010}]\label{TheoremRatioErgodic} Let $\{T^u\}_{u\in\mathbb R^d}$  be a free ergodic measure preserving action on a standard $\sigma$-finite measure space $X$.   Then for $\mu$-a.e. $x\in X$ and every $f,g\in L^1(X,\mu)$ with $\int_X g d\mu \neq 0 $, we have $$\frac{\int_{B_R}f(T^u (x))du }{\int_{B_R}g(T^u (x))du} \to \frac{\int_X f d\mu}{\int_X g d\mu}\ \ \mbox{as }R\to\infty.$$
\end{theorem}

\begin{remark}
In fact, \cite{Hochman:2010} considers non-singular free ergodic actions of $\Z^d$ or $\R^d$, which includes measure-preserving actions, and averaging is over balls
in any norm. We note that for our purposes it would suffice to use an older ratio ergodic theorem of M. Becker \cite{Becker:1983}, but it would require a little additional argument, so we chose to quote the recent more general result of M. Hochman.
\end{remark}

\begin{lemma}\label{LemmaFisherSecondOrder} Let $\{T^u\}_{u\in\mathbb R^d}$  be a free ergodic measure preserving action on a standard $\sigma$-finite measure space $(X,\mu)$.   Assume that there exists $\alpha>d-1$ such that for some function $g\in L^\infty(X,\mu)\cap L^1(X,\mu)$ with $\int_X g \,d\mu \neq 0$,  the limit
$$ \lim\limits_{t\to\infty}\frac{1}{\log(t)}\int_{1}^{t}\frac{\int_{B_R} g(T^u(x))\,du}{ (2R) ^{\alpha}}\,\frac{dR}{R}$$ exists and is finite for $\mu$ a.e. $x\in X$. Then

(i) this limit is constant almost everywhere;

(ii) writing this limit as $c\cdot \int_X g d\mu$ for some constant $c$, we get that for every function $f\in L^1(X,\mu)$,
 $$\lim\limits_{t\to\infty}\frac{1}{\log(t)}\int_{1}^{t}\frac{\int_{B_R} f(T^u(x))\,du}{c (2R) ^{\alpha}}\,\frac{dR}{R} = \int_X f\, d\mu$$ for $\mu$-a.e. $x\in X$.
\end{lemma}
\begin{proof} (i) Set $$S_R^g(x) = \int_{B_R} g(T^u(x))\,du\ \ \mbox{and}\ \
\overline g(x) = \lim\limits_{t\to\infty}\frac{1}{\log(t)}\int_{1}^{t}\frac{S_R^g(x)}{(2 R) ^{\alpha}}\,\frac{dR}{R}.$$ By  assumption, the function $\overline g(x)$ is finite for $\mu$-a.e.\ $x\in X$, and it is straightforward  to check that $\overline g$ is measurable. We claim that $\overline g(T^v(x)) = \overline g(x)$ for every $v\in \mathbb R^d$ and $\mu$-a.e. $x\in X$. Indeed,
$$
\int_{B_R} g(T^u(x))\,du - \int_{B_R} g(T^{u+v}(x))\,du = O(R^{d-1}\|g\|_\infty)\ \ \mbox{as}\ R\to \infty,
$$
and the assumption $\alpha>d-1$ yields the claim.  Since the measure $\mu$ is ergodic, we get that the function $\overline g$ is constant almost everywhere.

(ii) Applying Theorem \ref{TheoremRatioErgodic}, we obtain that $$\frac{S_R^f(x)}{S_R^g(x)}\to\frac{\int f\,d\mu}{\int g\,d\mu}\ \mbox{ as }R\to\infty.$$ Assume for definiteness that $\int g\,d\mu >0$ and $\int f\,d\mu \geq 0$. It follows that for any $\varepsilon>0$ there is $R_0>0$ such that for all $R>R_0$,
$$S_R^g(x)\frac{\int f\,d\mu}{\int g\,d\mu}\,(1-\varepsilon) \leq S_R^f(x) \leq S_R^g(x)\frac{\int fd\mu}{\int g\,d\mu}\,(1+\varepsilon).$$
Dividing the inequalities by $2^\alpha R^{\alpha+1}$ and integrating with respect to $R$ yields $$(1-\varepsilon)\int g\,d\mu\,\frac{\int f\,d\mu}{\int g\,d\mu} \leq \lim\limits_{t\to\infty}\frac{1}{\log(t)}\int_{1}^{t}\frac{S_R^f(x)}{
(2 R)^{\alpha}}\,\frac{dR}{R} \leq (1+\varepsilon)\int g\,d\mu\,\frac{\int f\,d\mu}{\int g\,d\mu}.$$
 Taking the limit as $\varepsilon\to 0$,  we obtain the result.
\end{proof}

Now we are ready to prove the main result of the paper, but first we give the definition of the average (or second-order) density, which was introduced by Bedford and Fisher \cite{BeFi}.

\begin{definition} \label{def-avden} Let $\alpha>0$ and $\eta$ a positive finite Borel measure on $\R^d$. The {\em average $\alpha$-dimensional density} of $\eta$ at $x$ is
$$
A^\alpha (\eta,x) = \lim_{k\to \infty} \frac{1}{k} \int_0^k \frac{\eta(B_{e^{-t}}(x))}{(2e^{-t})^\alpha}\,dt,
$$
if the limit exists. Note that we can replace $e^{-t}$ by $\lam^{-t}$ for $\lam>1$, without changing the value of $A^\alpha (\eta,x) $. For $d=1$, the right average density is defined as above, replacing $\eta(B_{e^{-t}}(x))/(2e^{-t})^\alpha$ by $\eta([x,x+e^{-t}))/e^{-\alpha t}$.
\end{definition}

It is known that for a graph-directed self-similar set $K$ satisfying the Open Set Condition, the $\alpha$-dimensional average density of the Hausdorff measure $\Hk^\alpha$ (where $\alpha$ is the dimension of the set) restricted to $K$, exists and is constant $\Hk^\alpha$-a.e. This is proved in \cite{BeFi} for srtandard IFS, and the extension to the graph-directed sets is straightforward (as we essentially show below).

%%%%%%%%%%%
%
% Main Theorem
%

\begin{theorem}\label{TheoremMainResult} Let $\mathbb X =  (\Omega_{\mathcal G},\mu,\R^d)$ be the  tiling dynamical system corresponding to a tile substitution $\mathcal G$. Suppose that the tiling system satisfies the assumptions of Section \ref{SectionAssumptions}. Assume that $\mu$ is an infinite ($\sig$-finite) invariant measure, positive and finite on $\Om_\B$, where $\Omega_{\mathcal B}$ is the set of tilings which have a type $\B$ tile containing the origin.

Then there exist positive parameters $\alpha$ and $c$ such that for $\mu$-almost every tiling $\mathcal T\in \Omega$ and for every function $f\in L^1(\Omega,\mu)$,  we have
\begin{equation}\lim\limits_{t\to\infty}\frac{1}{\log(t)}\int_{0}^{t}\frac{\int_{B_R} f(\T-u)\,du}{c (2R) ^{\alpha}}\,\frac{dR}{R}  = \int_{\Omega_{\mathcal G}} f d\mu.
\end{equation}
 Here $\alpha = \log(\rho(B))/\log(\lambda)$ is the Hausdorff dimension of the graph-directed self-similar sets from Section \ref{SubsectionGraphDirectedIFS} and
$$c = \gamma\cdot\lim\limits_{k\to\infty}\frac{1}{k}\int_{0}^k \frac{\mathcal H^{\alpha}( B_{\lambda^{-t}}(u)\cap K_Q)}{(2\lam^{-t})^\alpha}\, dt$$
for $\mathcal H^{\alpha}$-a.e.\ $u\in K_Q$ and for every $Q\in \mathcal B$, where $\Hk^\alpha$ is the $\alpha$-dimensional Hausdorff measure on $K_Q$ . The parameter $c$ is the average $\alpha$-dimensional density of $\Hk^\alpha$ restricted to $K_Q$, up to the normalizing constant $\gamma$:
$$
\gamma^{-1} = \sum_{Q\in \B} \xi_Q \Hk^\alpha(Q), \ \ \mbox{where}\ \ \sum_{Q\in \B} \xi_Q \Lk^d(Q)=\mu(\Om_\B),
$$
and $(\xi_Q)_{Q\in \B}$ is a right Perron-Frobenius eigenvector of the matrix $B$.
\end{theorem}

\begin{proof}  (1) First note that, without loss of generality, we can normalize $\mu$ in such a way that (\ref{EqNormal1}) holds, so that $\gam=1$. We then define $\nu$ as in Theorem \ref{ThConj} and consider the ergodic probability-preserving transformation $(\Omega_0,\nu,\mathcal G^{-1})$.
We follow, in part, the argument of \cite[Theorem 3.1]{BeFi} (see also \cite[Theorem 6.6]{Falconer:Techniques}).   Recall that $B_R$ denotes the closed ball of radius $R$  centered at the origin.
Define a function  $\psi: \Omega_0\rightarrow \mathbb R$ by $$\psi(\mathcal T) = \int\limits_{0}^1 \frac{\mathcal H^{\alpha}(\mathcal C(\mathcal T)\cap {B}_{\lambda^t})}{(2\lam^t)^\alpha}\,dt.$$
Since the Hausdorff measure $\mathcal H^{\alpha}$ is finite on sets $\mathcal K_T$, $T\in \mathcal B$, and the ball ${B}_\lambda$ contains only a finite number of tiles, the function  $\psi$ is bounded. It is straightforward to check that  the function $\psi$ is measurable.

 (2) Recall that  $\varphi = \lambda\cdot O$, where $O$ is an orthogonal matrix, hence $
\Hk^\alpha(\varphi^{-1}E) = \lam^{-\alpha} \Hk^\alpha(E)$ for any Borel set $E$. Note also that $\varphi(B_{\lambda^t}) = B_{\lambda^{t+1}}$.  Applying Proposition \ref{PropositionPropertiesGeometricSubstitution}, we obtain that
\begin{eqnarray*}\psi(\mathcal G^{-1}  (\mathcal T)) & = & \int\limits_{0}^1  \frac{\mathcal H^\alpha( \varphi^{-1}\mathcal( C(\mathcal T))\cap {B}_{\lambda^t})}{2^\alpha\lam^{t\alpha}}\, dt \\
%& = & \int\limits_{0}^1  \frac{1}{\lambda^d} \frac{\mathcal H^\alpha( \mathcal C(\mathcal T)\cap {B}_{\lambda^{t+1}})}{\lam^{td}}\,dt
\\
& = & \int\limits_{0}^1   \frac{\mathcal H^\alpha( \mathcal C(\mathcal T)\cap {B}_{\lambda^{t+1}})}{2^\alpha\lam^{(t+1)\alpha}}\,dt  \\
& = & \int\limits_{1}^2   \frac{\mathcal H^\alpha( \mathcal C(\mathcal T)\cap {B}_{\lambda^{t}})}{2^\alpha\lam^{t\alpha}}\, dt.
\end{eqnarray*}
It follows that $$ \sum_{i = 0}^{k-1}\psi(\mathcal G^{-i}\mathcal T) = \int\limits_{0}^k    \frac{\mathcal H^\alpha( \mathcal C(\mathcal T)\cap {B}_{\lambda^{t}})}{2^\alpha\lam^{t\alpha}}\, dt.$$
Thus, applying the Birkhoff Ergodic Theorem to the system $(\Omega_0,\nu,\mathcal G^{-1})$ and the function $\psi$, we get that
\begin{equation}\label{EquationErgodicTheoremGeometricSubstitution}\lim\limits_{k\to\infty }\frac{1}{k}\int\limits_{0}^k
\frac{\mathcal H^\alpha( \mathcal C(\mathcal T)\cap {B}_{\lambda^{t}})}{2^\alpha\lam^{t\alpha}}\,dt = \int_{\Omega_0}\psi(\mathcal S)\, d\nu(\mathcal S)\end{equation} for $\nu$-a.e.\ tiling $\mathcal T\in \Omega_0$.
Substituting $\lambda^t = R$ into (\ref{EquationErgodicTheoremGeometricSubstitution}), we obtain that

\begin{equation}\label{EquationErgodicTheoremTechn1}
\int_{\Omega_0}\psi(\mathcal S)\, d\nu(\mathcal S)  =  \lim\limits_{z \to\infty}\frac{1}{\log(z)} \int\limits_{1}^z \frac{\mathcal H^\alpha(\mathcal C(\mathcal T)\cap {B}_R)}{2^\alpha R^\alpha}\, \frac{dR}{R}  \end{equation} for $\nu$-a.e.\ $\mathcal T\in \Omega_0$. (Passing from $z=\lam^k$ for $k\in \Nat$ to an arbitrary $z>0$, $z\to \infty$, in the
limit above is justified, since $\psi$ is a bounded function.)

(3)
We have
$$
\Om_\G = \bigcup_{Q\in \A}(\Gam_Q - \supp(Q)).
$$
The sets in the right-hand side have intersections of zero $\mu$ measure in view of (\ref{EqProduct0}), since $\Lk^d(\partial(\supp(Q)))=0$.
Consider the function
$$
g:= \sum_{Q\in \B} \frac{\Hk^\alpha(Q)}{\Lk^d(Q)}\cdot \chi_{_{\Gam_Q - \supp(Q)}} \in L^\infty (\Om_\G, \mu).
$$
That is, $g(\Tk)$ is nonzero if and only if the origin lies in a $\T$-tile $Q-x$ of type $\B$, and then the value of the function is $\frac{\Hk^\alpha(Q)}{\Lk^d(Q)}$
(this is well-defined on a set of full $\mu$ measure). Then (\ref{EqProduct0}) implies
$$
\int_{\Om_\G}g(\T)\,d\mu(\T) = \sum_{Q\in \B} \mutr(\Gam_Q)\Hk^\alpha(Q)=1.
$$
In view of Lemma \ref{LemmaFisherSecondOrder}, it suffices to establish the second order ergodic theorem  just for the function $g$.

Given a tiling $\mathcal T\in \Omega_\mathcal G$, denote
$$
V_R(\T) = \int_{{B}_R} g(\T-u)\,du.
$$
Observe that for every tile $T=Q-x\in \T$, with $Q\in \B$, such that $\supp(T)\subset B_R$, integrating $g(\T-u)$ over $\supp(T)$ contributes
$\Hk^\alpha(K_T)$ to $V_R(\T)$. Exactly the same contribution from $T$ comes to $\Hk^\alpha(\C(\T)\cap B_R)$. Therefore, the difference between $\Hk^\alpha(\C(\T)\cap B_R)$
and $V_R(\T)$ is bounded (in modulus) by the sum of $\Hk^\alpha(K_T)$ over those $T\in \T$ of type $\B$ whose supports intersect $\partial B_R$.  Thus, denoting
by $d_M$ the maximal diameter of a prototile, we obtain
$$
\left|\mathcal H^{\alpha}(\mathcal C(\mathcal T)\cap B_R) - V_R(\mathcal T)\right|\leq
\Lk^d(B_R\setminus B_{R-d_{M}}) \cdot\max_{Q\in \B}\frac{\Hk^\alpha(Q)}{\Lk^d(Q)} = O(R^{d-1}),
$$
with the implied constant in $O(\cdot)$ depending only on the tiling $\T$.

 Since $\alpha > d-1$ (one of our standing assumptions), we obtain that
 $$\lim\limits_{z \to\infty}\frac{1}{\log(z)} \int\limits_{1}^z \frac{|\mathcal H^\alpha(\mathcal C(\mathcal T)\cap B_R) - V_R(\mathcal T)|}{2^\alpha R^{\alpha+1}}\, dR =0.
$$
Using Equation (\ref{EquationErgodicTheoremTechn1}), we conclude that
\begin{equation}\label{EquationSecOrderTechnical2}\lim\limits_{z \to\infty}\frac{1}{\log(z)} \int\limits_{1}^z \frac{\int_{B_R}g(\mathcal T - u)du}{2^\alpha R^{\alpha+1}}\, dR = \int_{\Omega_0}\psi(\mathcal S)\,d\nu(\mathcal S)\end{equation} for $\nu$-a.e.\ tiling $\mathcal T \in \Omega_0$. Denote by $Y$ the set of all tilings in $\Omega_\mathcal G$ for  which Equation (\ref{EquationSecOrderTechnical2})  holds. Observe that if  $\mathcal T \in Y$, then  $\mathcal T - v \in Y$ for any $v\in \mathbb R^d$ (here we use $\alpha>d-1$ again), i.e.\ $Y$ is  translation-invariant.
Translation invariance of $Y$ implies that  $Y\cap (\Gam_Q - K_Q) = (Y\cap\Gam_Q) - K_Q$ for each prototile $Q\in \mathcal B$.
Since $\nu(Y\cap \Om_0 ) >0$,  there is a prototile $Q\in \mathcal B$ with $\nu(Y\cap (\Gam_Q - K_Q)) > 0$.
Then Theorem~\ref{ThConj}(ii)  implies that $$0 < \nu(Y\cap (\Gam_Q - K_Q)) = \nu((Y\cap \Gam_Q) - K_Q) = \mutr(Y\cap \Gam_Q)\cdot\Hk^\alpha(K_Q),$$
hence $\mutr(Y\cap \Gam_Q)>0$. Again using translation-invariance of $Y$ and (\ref{EqProduct0}) we obtain
$$
\mu(Y) \geq \mu((Y\cap \Gam_Q)-\supp(Q))= \mutr(Y\cap \Gam_Q)\cdot\Lk^d(\supp(Q))>0,
$$
and ergodicity of the tiling dynamical system $\Om_\G,\mu,\R^d)$ implies that
(\ref{EquationSecOrderTechnical2}) holds for $\mu$-a.e. $\mathcal T\in \Omega_{\mathcal G}$. Setting $$c = \int_{\Omega_0}\psi(\mathcal S)\, d\nu(\mathcal S),$$ we get the result.

(4) It remains to show that the parameter $c$ can be interpreted as the average density of the Hausdorff measure on the graph-directed set.  Using the same arguments as in (2) above, we obtain that \begin{eqnarray*}\psi(\mathcal G^k(\mathcal T)) &= &\int\limits_{0}^1  \frac{\mathcal H^\alpha( \varphi^k(\mathcal C(\mathcal T))\cap B_{\lambda^t})}{2^\alpha \lam^{t\alpha}}\, dt \\
& = & \int\limits_{0}^1  \frac{\mathcal H^\alpha( \mathcal C(\mathcal T)\cap B_{\lambda^{t-k}})}{2^\alpha\lam^{(t-k)\alpha}}\, dt  \\
& = & \int\limits_{-k}^{-k+1}  \frac{\mathcal H^\alpha( \mathcal C(\mathcal T)\cap B_{\lambda^{t}})}{2^\alpha\lam^{t\alpha}}\, dt .  \end{eqnarray*}
Applying the Birkhoff Ergodic Theorem to the system $(\Omega_0,\nu,\mathcal G)$ and the function $\psi$, we see that   for $\nu$-a.e.\ $\mathcal T\in \Omega_0$,
 \begin{equation}\label{EquationSecOrderTechnical3}\lim\limits_{k\to\infty}\frac{1}{k}\int_{-k}^0 \frac{\mathcal H^\alpha(\mathcal C(\mathcal T)\cap B_{\lambda^t})}{2^\alpha\lam^{t\alpha}}\, dt = \lim\limits_{k\to\infty}\frac{1}{k}\int_{0}^k \frac{\mathcal H^\alpha(\mathcal C(\mathcal T)\cap B_{\lambda^{-t}})}{2^\alpha\lam^{-t\alpha}}\, dt = c.\end{equation}
We have
$$
\Om_0 = \bigcup_{Q\in \B} (\Gam_Q - K_Q),
$$
and the sets in the right-hand side have intersections of zero $\nu$ measure, in view of Theorem~\ref{ThConj}(ii) and (\ref{EqDis}).
Then the set of tilings $\T\in \Om_0$, such that (\ref{EquationSecOrderTechnical3}) holds and $\T$ belongs to $\Gam_Q-\K_Q$ for a unique $Q\in \B$,
has full $\nu$ measure. Denote this set of tilings by $Z$. Observe that the behavior of the limit in (\ref{EquationSecOrderTechnical3}) depends only
on the small neighborhood of the origin (because the limit doesn't change if we replace $\int_0^k$  by $\int_i^k$ for any fixed $i\in \Nat$), hence
$\T-u\in Z$ for $T\in \Gam_Q$ and $u\in K_Q$ implies $\T'-u\in Z$ for any $\T'\in \Gam_Q$. Thus, for every $Q\in \B$,
$$
Z\cap (\Gam_Q-K_Q) = \Gam_Q -K_Q'
$$
for some $K_Q'\subset K_Q$. We have
\begin{eqnarray*}
\mutr(\Gam_Q)\cdot \Hk^\alpha(K_Q) & = &\nu(\Gam_Q-K_Q)\\& =& \nu(Z\cap (\Gam_Q-K_Q) )\\ &=& \nu(\Gam_Q-K_Q')\\&= &\mutr(\Gam_Q)\cdot \Hk^\alpha(K_Q').
\end{eqnarray*}
It follows that $\Hk^\alpha$-a.e.\ $u\in K_Q$ is such that $\T-u\in Z$ for $T\in \Gam_Q$, which means, rewriting (\ref{EquationSecOrderTechnical3}), that
$$c = \lim\limits_{k\to\infty}\frac{1}{k}\int_{0}^k \frac{\mathcal H^\alpha( K_Q\cap B_{\lambda^{-t}}(u))}{2^\alpha \lam^{-t\alpha}}\, dt\  \ \mbox{for\  $\Hk^\alpha$-a.e.\ $u\in K_Q$,\ for all\ $Q\in \Bk$},$$ as desired. The proof is complete.
\end{proof}

\begin{remark}\label{RemAver}
1. Since there exists $f\in L^1(\Om_\G,\mu)$ such that $0<\int_{\Om_\G} f\,d\mu <\infty$, it is immediate that the
paramaters $\alpha$ and $c$ in Theorem~\ref{TheoremMainResult} are invariants of measure-theoretic
isomorphism.

2. We considered averaging over Euclidean balls in Theorem~\ref{TheoremMainResult}. This was needed in the equality $\varphi(B_R)=B_{\lam R}$. If we restrict ourselves to the case when $\varphi$ is a pure dilation, i.e.\ $\varphi(x) = \lam x$ for $\lam>1$, then we can use averaging over balls in any norm.
\end{remark}

%%%%%%%%%%%%%%%%%%%%%%%%%%%%%%%%%%%%%%%%%%%%%%
%
%
% Substitution dynamical systems

\section{Substitution  Dynamical Systems}\label{SectionSubstitutionDynamicalSystems}

In this section, we derive the second order ergodic theorem for a class of one-dimensional substitution systems. We  begin with a brief review of the
background. This will be reminiscent of our discussion of the structure of tiling substitutions, however, there are certain fundamental differences.
 One of the principal differences between symbolic substitution systems and their tiling counterparts is that  symbolic substitutions may have finite ergodic invariant measures supported off the minimal components.

Now $\mathcal A$   is a finite alphabet (usually $\A = \{1,\ldots, N\}$) and $\mathcal A^+$ is the set of all finite non-empty words over
$\mathcal A$. A map $\sigma: \mathcal A\rightarrow \mathcal A^+$ is called a {\it substitution}; it is extended to $\mathcal A^+$ by  concatenation.  Given two words $v,w\in \mathcal A^+$, we will write $v \prec w$ if $v$ is a subword  of $w$.
 Denote by $L(\sigma)$ the set of all words $w\in \mathcal A^+$ such that  $w\prec \sigma^n(a)$ for some  $a\in \mathcal A$ and $n\geq 1$. The family $L(\sigma)$ is called the {\it language} of the substitution.  For a word $w$, its length will be denoted by $|w|$.

\begin{definition}  The {\it substitution dynamical system} determined by a substitution $\sigma$ is a pair $(X_\sigma,S)$, where  $$X_\sigma = \{x\in \mathcal A^\mathbb Z :\  x[-n,n] \in L(\sigma)\mbox{ for all }n\geq 1\}$$  and
 $S:\mathcal A^\mathbb Z\rightarrow \mathcal A^\mathbb Z$ is the left  shift.
The set $X_\sigma$ is $S$-invariant and  closed in  $\mathcal A^\mathbb Z$ with respect to the product topology.
\end{definition}

Given  $a,b\in \A$, denote by $m_{a,b}$ the number of occurrences of $a$ in the word $\sigma(b)$. The matrix $M_\sigma = (m_{a,b})_{a,b\in \mathcal A}$  is  the {\it substitution matrix} of $\sigma$.  Reordering the letters in the alphabet $\mathcal A$, the  matrix $M_\sigma$ can be transformed to have  an upper  block-triangular form as in (\ref{FrobeniusForm}). The diagonal matrices $F_i$ determine the structure of invariant subsets and the spectral properties of $F_i$ determine the structure of invariant measures.  We will not discuss these results here and  refer the reader to  \cite{BezuglyiKwiatkowskiMedynetsSolomyak:2010} for details.

Our approach to the second order ergodic theorem of substitution systems is based on Theorem~\ref{TheoremMainResult}, which we apply to the self-similar substitution
system on the line $\R$, arising from the symbolic substitution system $(X_\sig,S)$ as a suspension flow. For this to exist, however, it is necessary and
sufficient that the substitution matrix $M_\sig$ should have a strictly positive left eigenvector, whose components will serve as the lengths of the prototiles. This is a significant restriction: it is known that a strictly positive left eigenvector for the matrix $M_\sig$ in the form (\ref{FrobeniusForm}) exists if and only if all the
matrices $F_1,\ldots,F_s$, corresponding to the minimal components, have the same spectral radius, which is strictly greater than the spectral radii of all the
remaining diagonal blocks $F_{s+1}, \ldots, F_m$, see \cite[Th.III.6,\,p.92]{Gant}. 

\medskip

\noindent{\em Standing assumption.}
We will assume for simplicity that
 \begin{equation}\label{EqSubstMatrixNormalForm}
M_\sigma =\left(
  \begin{array}{cc}
    A & C \\
   0 & B
  \end{array}
\right),
\end{equation}
where $A$ and $B$ are primitive, with $\rho(A)> \rho(B)>1$, and $C$ is non-zero.  We expect that our method works in the more general case, when $M_\sig$
has a strictly positive left eigenvector, but we have not verified the details.

Under our standing assumption, plus a technical condition stated below,
the system $(X_\sig,S)$ has a unique, up to scaling, invariant measure that is positive and finite on at least one open set, and
this measure is infinite $\sig$-finite. This follows from Corollary 5.6 in \cite{BezuglyiKwiatkowskiMedynetsSolomyak:2010} in the case when the
substitution system is non-periodic. The non-periodicity was needed to ensure the recognizability property \cite[Theorem 5.17]{BezuglyiKwiatkowskiMedynets:2009},  however, it is possible to extend  the proof of  Theorem 5.17 from \cite{BezuglyiKwiatkowskiMedynets:2009} to the needed generality. 

\medskip

 First we need a technical lemma.
 Given two sequences $\{x_n\}$ and $\{y_n\}$ of reals, the notation $x_n\approx y_n$  means that $x_n/y_n\to 1$ as $n\to\infty$.

\begin{lemma} \label{LemmaSubstitutionsWordsGrowth} Let $\mathcal A = \{1,2,\ldots,N\}$ be a finite alphabet and $\sigma: \mathcal A \rightarrow \mathcal A^+$
a substitution with the substitution matrix of the form (\ref{EqSubstMatrixNormalForm}).
Assume that the matrices $A$ and $B$ are primitive and $\rho(A) > \rho(B)>1$.  Then
\beq \label{eqasymp1}
|\sig^k(i)|\approx \xi_i \rho(A)^k,\ \ i=1,\ldots,N,
\eeq
where $\ov{\xi} = (\xi_i)_{i=1}^N$ is a left Perron-Frobenius eigenvector for $M=M_\sig$, i.e.\
$$[\xi_1\ldots \xi_N]M= \rho(A) [\xi_1\ldots \xi_N].
$$
\end{lemma}

\begin{proof}
We have $|\sig^k(i)| = \langle M^k \be_i, \o\rangle$ where $\be_i$ is the $i$-th unit vector and $\o = [1\ldots 1]^t$.
Asymptotics of the entries of powers of a non-negative (not necessarily irreducible) matrix are known.
It follows e.g.\ from Theorem (9.4) in \cite{Schneider} that there exist $\xi_i>0$ such that $|\sig^k(i)|\approx \xi_i \rho(A)^k$, $i\le N$.
It remains to show that $\overline \xi = \{\xi_i\}$ is a left eigenvector.

 Notice that
$$
M^{k+1}\be_j = \sum_{i=1}^N M(i,j) M^k \be_i.
$$
Hence
$$
\langle M^{k+1}\be_j,\o\rangle = \sum_{i=1}^N \langle M^k \be_i,\o\rangle M(i,j),
$$
which implies
$$
\xi_j \rho(A)^{k+1} \approx\sum_{i=1}^N \xi_i \rho(A)^k M(i,j)\mbox{ as }k\to\infty.
$$
This implies that  $\ov{\xi}$ is a  left eigenvector for $M$, as desired.
\end{proof}

\begin{definition}\label{DefinitionTilingAssociatedSubstitution} Set $\lambda = \rho(A)$. Let $\overline \xi$ be the left eigenvector for the matrix $M_\sig$  in the lemma above, satisfying (\ref{eqasymp1}).
For each letter  $a\in \mathcal A$, denote by $I_a$ the interval of length $\xi_a$ centered at the origin. We will consider these intervals as tiles in $\mathbb R$,
labeled by their letters. Set $\varphi(x) = \lambda x$.
Define the tile substitution $\mathcal G$ on the tiles $\{I_a\}_{a\in \A}$  as follows. Consider the inflated  tile $\varphi (I_v)=\lam I_v$. Since $$\textrm{length}(\lam I_v) = \lambda \xi_v = \sum_{w\in \mathcal A}M(w,v)\xi_w,$$ we can subdivide the interval $\lam I_v$ into the intervals $\{I_w\}$ according to the sequence of all the letters of $\sigma(v)$. Define $\mathcal G(I_v)$ as the collection of the corresponding translates of the intervals $\{I_w\}_{w\in \sigma(v)}$.
We will call $\mathcal G$ the {\it tile substitution associated with $\sigma$} and denote by $\Om_\G$ the corresponding tiling space.
\end{definition}

 Denote by $\B$ the set of letters corresponding to the matrix $B$.

\begin{lemma} \label{LemSusp}
(i) The tiling substitution $\R$-action $(\Om_\G,\R)$ is isomorphic (canonically topologically conjugate) to the suspension flow over the symbolic substitution
$\Z$-action $(X_\sig,S)$, with the ``roof function'' equal to the constant $\xi_j>0$ on the cylinder sets $[j]$ for $j\in \A$.

(ii) Assume that the substitution $\G$ satisfies the  conditions of Section 2.3. Then there is a unique infinite ($\sig$-finite) invariant measure $\nu$ for the system $(X_\sig,S)$, normalized so that
$$
\sum_{b\in \B} \xi_b \nu([b])=1.
$$
This measure may be identified with the transverse measure $\mutr$ of the invariant measure $\mu$ for the system $(\Om_\G,\R)$, normalized so that
$\mu(\Om_\B)=1$, where $\Om_\B$ is the set of tilings from $\Om_\G$ having a tile of type $\B$ containing the origin.
\end{lemma}
\begin{proof}
This follows from definitions and the results of \cite{CortezSolomyak:2011}. We just observe that the transversal of $\Om_\G$ may be naturally identified with
$X_\sig$, and transversal measures correspond to invariant measures for $(X_\sig,S)$.
\end{proof}

For the technical assumptions from Section 2.3 to hold, it is enough that
\beq \label{tec1}
\forall\,b\in \B,\ \sig(b)\ \mbox{starts and ends with a letter from}\ \B,
\eeq
and
\beq \label{tec2}
\forall\,b\in \B,\ \exists\,k\in\Nat\ \mbox{such that}\ \sig^k(b)\ \mbox{has at least one ``interior'' letter from}\ \B.
\eeq

\medskip

\begin{definition} \label{DefAdmis}
We will call a substitution $\sig$ {\em admissible} if it has the form  (\ref{EqSubstMatrixNormalForm}), with $\rho(A)>\rho(B)>1$, and both
(\ref{tec1}) and (\ref{tec2}) are satisfied.
\end{definition}

\begin{remark}
Actually, condition (\ref{tec1}) may be omitted: it implies the ``non-periodic border condition'', see Definition~\ref{DefNonper}, which was needed for
recognizability of non-periodic tilings. In fact, in the setting of one-dimensional self-similar tiling substutions, the proof of recognizability from \cite{CortezSolomyak:2011} works without it.
\end{remark}

\medskip

Let $(X,T,\nu)$ be an infinite ergodic measure-preserving transformation. The system is called {\it conservative} if it has no wandering sets of positive measure, i.e.\ there is no set $W\subset X$  with $\nu(W) >0 $ and $W\cap T^{-n}W = \emptyset$ for every $n\geq 1$.
We need conservativity of our systems, since we will consider one-sided averages for the substitution $\Z$-action.

\begin{lemma} \label{LemConserv} The substitution dynamical system with an infinite invariant measure $(X_\sigma,\nu,S)$ ,
corresponding to an admissible substitution $\sig$, is conservative.
\end{lemma}

\begin{proof}
We will use Maharam's recurrence theorem (see \cite[1.1.7]{Aaronson:book}), which says that if there exists a subset $Y$ of finite measure, such that
$X_\sig = \bigcup_{n=0}^\infty S^{-n} Y$ mod $\nu$, then $S$ is conservative. Let $Y$ be the set of sequences $(y_n)_{n\in \Z}\in X_\sig$ such that $y_1\in \B$.
Then $\nu(Y)<\infty$ by Lemma~\ref{LemSusp}. We have $y\not\in \bigcup_{n=0}^\infty S^{-n} Y$  if and only  if  there exists $k\in \Z$ such that $y_n \in \A\setminus \B$ for all $n>k$. Since $\nu$ is supported on the set of
sequences which contain at least one $\B$-symbol, it suffices to show that
$$
\nu(Y_0) = 0,\ \ \mbox{where}\ \ Y_0=\{x\in X_\sig:\ x_n \in \A\setminus\B\ \mbox{for all}\ n> 0\ \mbox{and}\ x_0\in \B\}.
$$
For every $y\in Y_0$ and $n>0$ there exist $b\in \mathcal B$ and $i=0,\ldots,|\sigma^n(b)|-1$ such that  $y \in S^i[\sigma^n(b)]$,
 where $[\sigma^n(b)] = \{x\in X_\sigma : x[0,|\sigma^n(b)|-1] = \sigma^n(b)\}$. Since every word $\sigma^n(b)$, $b\in \mathcal B$, ends with a letter from $\mathcal B$, we immediately get that $i = |\sigma^n(b)|-1$.
It follows that  $Y_0 \subset \bigcup_{b\in \B} \bigcap_{n\ge 1} S^{|\sigma^n(b)|-1}[\sig^n(b)]$.

Since the measure $\nu$ is non-atomic, we have that
 $$\nu(S^{|\sigma^n(b)|-1}[\sig^n(b)]) = \nu([\sig^n(b)])\to 0\mbox{ as }n\to \infty.$$ This yields the result.
\end{proof}

The following simple ``folklore''  lemma gives the so-called ``accordion'' representation of  words from $X_\sigma$.

\begin{lemma} Let $x\in X_\sigma$ and $n\geq 1$. Then \begin{equation} \label{eq-accord}
x[1,n]=u_0 \sigma(u_1)\sigma^2(u_2)\ldots \sigma^m(u_m)\sigma^m(v_m) \sigma^{m-1}(v_{m-1})\ldots \sigma(v_1)v_0,
\end{equation}
where $m\ge 1$ and $u_i,v_j,\ i,j=1,\ldots,m$, are    subwords (possibly empty) of $\sigma(a)$, $a\in \mathcal A$. Moreover, at least one of $u_m, v_m$ is
nonempty.
\end{lemma}
\begin{proof}
Set $w = x[1,n]$. By the definition of $X_\sig$, we can choose $a\in \A$ and the minimal $k\in \N$ such that $w\prec\sig^k(a)$. Writing
$\sig^{k-1}(a) = a_1\ldots a_m$ we obtain that $w\prec \sig(a_1)\ldots \sig(a_m)$, hence
$$
w =u_0\sig(w^{(1)})v_0,
$$
where $w^{(1)}$ is a subword $\sig^{k-1}(a)$ (possibly empty), $u_0$ is a suffix of some $\sig(a_i)$ (possibly empty), and $v_0$ is a prefix of some
$\sig(a_j)$ (possibly empty). Repeating this process with $w^{(1)}$, etc., by induction, we obtain the desired representation (\ref{eq-accord}).
\end{proof}

For a word $w\in \A^+$ we define its ``population vector'' by
$\ov{\ell}(w) = (\ell_i(w))_{i=1}^N$ where $\ell_i(w)$ is the number of symbols $i$ in the word $w$.
For $w\in \A^+$ denote
\begin{equation} \label{eq-tile}
|w|_\Tk := \langle \ov{\ell}(w),\ov{\xi}\rangle
\end{equation}
and call this quantity the {\it tiling length} of the word $w$. Note that $\ov\ell(\sig(w))=M_\sig\ov\ell(w)$ by  definition of the
substitution matrix.

\begin{lemma} \label{LemmaWordTilingLength}
Let $\sig$ be an admissible substitution. Then for any $x\in X_\sigma$ we have
$$
\lim_{n\to \infty} \frac{|x[1,n]|_\Tk}{n} =1.\
$$
\end{lemma}

\begin{proof} Given $x\in X_\sig$ and $n\ge 1$, consider the accordion representation (\ref{eq-accord}). Note that  for all $i$,
$$
|u_i|, |v_i| \leq \max_a |\sig(a)|=:L_{\max}.
$$
Recall that for each letter $a\in \Ak$ we have $|\sig^k(a)|\approx \xi_a \lam^k=\lam^k|a|_\T$ by Lemma~\ref{LemmaSubstitutionsWordsGrowth}.
 Hence \begin{equation} \label{eq-tile2}
|\sig^k(u)|\approx \lam^k|u|_\Tk,\ \ k\to \infty,
\end{equation}
uniformly for all $u$ with $|u| \leq L_{\max}$. Note also
$$
\langle \ov\ell(\sig^k(j)),\ov\xi\rangle = \langle M_\sig^k\be_j,\ov\xi\rangle = \langle \be_j,(M_\sig^t)^k \ov\xi\rangle=
\langle \be_j,\lam^k\ov\xi\rangle = \lam^k\xi_j,
$$ which yields $|\sigma^k(u)|_\Tk = \lambda^k |u|_\Tk$.
Using the accordion representation of $x[1,n]$, we obtain
$$
\frac{|x[1,n]|_\Tk}{n} =
\frac{\sum_{i=0}^m \lam^i|u_i|_\Tk + \sum_{i=0}^m \lam^i |v_i|_\Tk}{\sum_{i=0}^m |\sig^i(u_i)|+ \sum_{i=0}^m |\sig^i(v_i)|}\,.
$$
Now the desired statement follows from (\ref{eq-tile2}) and the fact that $m\to \infty$ as $n\to \infty$ and at least one of $u_m,v_m$ is nonempty.
\end{proof}

The next lemma gives an upper bound for the number of $\mathcal B$-tiles in the interval $[0,t]$.
Recall that $\Omega_\G$ is the tiling space of $\mathcal G$, the tile substitution associated with $\sigma$. Given $\mathcal T\in \Omega_\G$, denote by
$N_\T(\B,t)$ the total number of $\B$-tiles of $\T$, contained (completely) in the interval $[0,t]$.

\begin{lemma}\label{LemmaBoundBetaWords} There exists a constant $K>0$ such that for every $\mathcal T\in \Omega_\G$ and  $t>0$,  we have
$N_\T(\B,t)\le Kt^\alpha$, where $\alpha = \log(\rho(B))/\log(\rho(A))$.
\end{lemma}
\begin{proof}  Since the inequality will persist (with a slightly larger constant $K$) if we shift $\T$ by a fixed vector, we can assume that $\T$
belongs to the transversal of $\Om_\G$.
For every integer $s>0$, find a tiling $\mathcal T_s$ such that $\mathcal G^s(\mathcal T_s) = \mathcal T$, see Proposition \ref{PropositionTilingSurjection}. Choose   an integer $k>0$ such that $\rho(A)^k\xi_i >2$  for every $i\in \mathcal A$. Let $T_s$ be the tile of $\T_s$ containing the origin; it is centered at
the origin for all $s$: since $\T$ is in the transversal, all $T_s$ are actual prototiles, see Definition~\ref{DefinitionTilingAssociatedSubstitution}. Then
the interval $[0,\rho(A)^{s}]$ is covered by the patch $\mathcal G^{s+k}(T_s)$ by our choice of $k$. Thus, $N_\T(\B,\rho(A)^{s})$ does not exceed the number of occurrences of $\mathcal B$-tiles in the patch $\mathcal G^{s+k}(T_s)$. By the Perron-Frobenius theorem applied to the primitive matrix $B$, the number of $\mathcal B$-tiles in  $\mathcal G^{s+k}(T_s)$   asymptotically  grows not faster than  $K\rho(B)^s$ for some constant $K$ independent of $s$.
The constant $K$ can be adjusted so that $N_\T(\B,\rho(A)^s)\leq K \rho(B)^s$ for every  tiling $\mathcal T$ and every positive real number $s$.
Setting $t = \rho(A)^s$, noting that
 $\rho(A)^\alpha = \rho(B)$ and adjusting the constant again,  we obtain the desired inequality for all $t>0$.
\end{proof}

Now we are ready to prove the main result on substitutions. It will be convenient to write elements of $X_\sig$ as $(x(n))_{n\in \Z}$.

%%%%%%%%
%
%
\begin{theorem}\label{TheoremSecondOrderSubstitutions} Let $\sigma :\mathcal A\rightarrow \mathcal A^+$ be an admissible substitution, see Definition~\ref{DefAdmis}. Let $\nu$ be the infinite invariant measure on $X_\sig$ from Lemma~\ref{LemSusp}.

Then  for every function $f\in L^1(X_\sigma,\nu)$ and $\nu$-a.e. $x\in X_\sigma$, we have that
$$\int_{X_\sigma}f(y)d\nu(y)  =  \lim\limits_{n\to\infty}\frac{1}{\log(n)}\sum_{k=1}^n\frac{\sum_{i=0}^{k-1}f(S^ix)}{c k^\alpha}\,\frac{1}{k}\,,
$$ where $\alpha = \log(\rho(B))/\log(\rho(A))$ and $c>0$ is the right average  $\alpha$-dimensional density of $\Hk^\alpha$ restricted to any of the graph-directed sets associated with $\mathcal G$.
\end{theorem}
\begin{proof}  (1) Consider the one-dimensional tile substitution system $(\Omega_\G,\mathbb R)$ associated with $\mathcal G$, see Definition \ref{DefinitionTilingAssociatedSubstitution} and Lemma~\ref{LemSusp}. Denote by $\mu$ the (unique) translation-invariant measure on $\Omega_\G$ corresponding to $\nu$. Recall that $X_\sig$ can be identified with the transversal of $\Om_\G$. More precisely, for $x\in X_\sig$, let $\T(x)\in \Om_\G$ be the tiling, which has the
prototile $I_{x(0)}$ as its tile, and the other tiles are $I_{x(n)}+y_n$, so that the left endpoint of $I_{x(n+1)}+y_{n+1}$ is the right endpoint of
$I_{x(n)} + y_n$ for all $n\in \Z$.

(2) In view of Theorem 4 in \cite{Fisher:1992}, which is a one-sided version of Lemma \ref{LemmaFisherSecondOrder},  it is enough to establish the result  for a single  function $f\in L^1(X_\sigma,\nu)$ with $\int fd\nu\neq 0$.  Theorem 4 from \cite{Fisher:1992} was established under the assumption  that the system is  conservative, so we use  Lemma~\ref{LemConserv} here. Consider the function $f$ on $X_\sigma$ given by: $f(x) = 1$ if $x(1)\in \mathcal B$ and $f(x) = 0$  otherwise.

Let $F$ be the function on $\Omega_\G$ such that $F(\mathcal T) = 1/\xi_i$ if  the tile of $\mathcal T$ containing the origin is a translate of $I_i$ for some  $  i\in \mathcal B$, and $F(\mathcal T) = 0$ otherwise. This is well-defined $\mu$-a.e.
Repeating the arguments of Theorem \ref{TheoremMainResult}, we obtain that for $\mu$-a.e.\ tiling $\mathcal T\in \Omega_\G$,
\begin{equation}\label{EqutionSubstituionsTech1} \lim\limits_{t\to\infty}\frac{1}{\log(t)}\int_1^t\frac{\int_0^R F(\mathcal T - u)\,du}{ R^{\alpha+1}}\, dR  = c\int_{\Omega_\G}F(\mathcal S)\,d\mu(\mathcal S):=\theta>0,\end{equation}
where $\alpha =\log(\rho(B))/\log(\rho(A))$ and $c>0$ is the right average  $\alpha$-dimensional density of $\Hk^\alpha$ restricted to any of the graph-directed sets associated with $\mathcal G$

Recall that $N_\T(\B,R)$ is the number of $\B$-tiles contained in $[0,R]$. Thus,
 $$ N_\T(\B,R) \le \int_0^RF(\mathcal T-u)\,du \le N_\T(\B,R)+1.$$
Since $\mu$   is (locally) a product of the transverse measure $\nu$ and the Lebesgue measure on $\mathbb R$, it follows from
 Equation (\ref{EqutionSubstituionsTech1}) that for $\nu$-a.e.\ $x\in X_\sigma$,

\begin{equation}\label{EqutionSubstituionsTech2} \lim\limits_{t\to\infty}\frac{1}{\log(t)}\int_0^t\frac{N_{\T(x)}(\B,R)}{R^{\alpha+1}}\,dR = \theta.\end{equation}

(3) We want to show that
$$
\theta = \lim_{n\to \infty} \frac{1}{\log(n)}\sum_{k=1}^{n}\frac{\sum_{i=0}^{k-1} f(S^i x)}{k^{\alpha+1}}
$$
for all $x\in X_\sig$ satisfying Equation (\ref{EqutionSubstituionsTech2}). Denote by $\ell_\B(w)$ the number of $\B$-letters in a word $w\in \A^+$, and observe
that
$$
\sum_{i=0}^{k-1} f(S^i x)=\ell_\B(x[1,k]).
$$
Note that  (\ref{EqutionSubstituionsTech2}) is equivalent to
$$
 \lim_{n\to \infty} \frac{1}{\log(n)}\sum_{k=1}^{n}\frac{N_{\T(x)}(\B,k)}{k^{\alpha+1}} = \theta,
$$
so we just need to estimate $|N_{\T(x)}(\B,k)-\ell_\B(x[1,k])|$. We claim that
\beq \label{compare}
\ell_\B(x[1,k]) = N_\B(\T(x),R_k),\ \ \ \mbox{where}\ \ R_k = |x[1,k]|_\T + (1/2)\xi_{x(0)}.
\eeq
Indeed, the left-hand side represents the number of $\B$-letters in $x[1,k]$ and the right-hand side equals the number of $\B$-tiles, (completely) contained in the interval $[0,R_k]$. According to the definition of $\T(x)$ at the beginning of the proof, it has the prototile $I_{x(0)}$ centered at the origin, so that half of its
length $(1/2)\xi_{x(0)}$ is in $[0,R_k]$. After that the sequence of tiles which fits in $[0,R_k]$ exactly corresponds to $x[1,k]$, by the definition of the tile
length. Thus, both sides of (\ref{compare}) count the same quantity.

Now, by Lemma~\ref{LemmaWordTilingLength} we have $R_k\approx k$, hence $R_k=k+o(k)$, as $k\to\infty$, using the standard $o(\cdot)$ notation.  In view of
(\ref{compare}) and Lemma~\ref{LemmaBoundBetaWords}, we have
\begin{eqnarray*}
|N_{\T(x)}(\B,k)-\ell_\B(x[1,k])| &=& |N_{\T(x)}(\B,k)-N_{\T(x)}(\B,R_k)|\\  & \le & K|R_k-k|^\alpha+1 = o(k^\alpha).
\end{eqnarray*}
Indeed, $N_{\T(x)}(\B,R_k)-N_{\T(x)}(\B,k)$ equals the number of $\B$-tiles of $\T(x)$ in the interval $[k,R_k]$ (assume that $k\le R_k$ for definiteness) plus
one, if a $\B$-tile contains $k\in \R$ in its interior, and the number of $\B$-tiles of $\T(x)$ in the interval $[k,R_k]$ equals $N_\B(\T(x)-k, R_k-k)$, to which we
can apply Lemma~\ref{LemmaBoundBetaWords}.
Thus,
\begin{eqnarray*}
\lim_{n\to \infty} \frac{1}{\log(n)}\sum_{k=1}^{n}\frac{\sum_{i=0}^{k-1} f(S^i x)}{k^{\alpha+1}} & = &
\lim_{n\to \infty} \frac{1}{\log(n)}\sum_{k=1}^{n}\frac{\ell_\B(x[1,k])}{k^{\alpha+1}} \\
& = & \lim_{n\to \infty} \frac{1}{\log(n)}\sum_{k=1}^{n}\frac{N_{\T(x)}(\B,k)}{k^{\alpha+1}} \\ & + & \lim_{n\to \infty} \frac{1}{\log(n)}\sum_{k=1}^{n}\frac{o(k^\alpha)}{k^{\alpha+1}} = \theta + 0,
\end{eqnarray*}
as desired.
Noticing that $$\theta = c \int_{\Omega_\G} F\,d\mu = c\sum_{b\in \mathcal B}\nu([b]) = c\int_{X_\sigma}f\,d\nu $$
by Lemma~\ref{LemSusp}, we get the result.
\end{proof}

\begin{remark} (1) We note that the parameters $\alpha$ and $c$ appearing in the second order ergodic theorem are invariants of the measure-theoretical isomorphism between infinite measure preserving systems.
This is immediate, since there exists $f\in L^1(X_\sig,\nu)$ such that $\int_{X_\sig} f d\nu$ is positive
and finite.

As an example, consider two symbolic substitution systems on the alphabet $\mathcal A = \{0,1\}$ given by $\sigma_1(0) = 0^3$ (three zeros), $\sigma_1(1) = 101$; and $\sigma_2(0) = 0^9$, $\sigma_2(1) = 1^401^4$. Then, $\alpha_1 = \log(2)/\log(3)$, whereas $\alpha_2 = \log(8)/\log(9)$. Since $\alpha_1\neq \alpha_2$, these systems cannot be measure-theoretically isomorphic with respect to the invariant infinite measures, and hence, cannot be topologically conjugate.

The parameter $c$ can also be used to distinguish substitution systems, although the computation is more involved. For
example, consider for $k=0,\ldots,3$ the substitutions $\sig_k(0)=0^9$ and $\sig_k(1) = 10^k10^{6-k}1$. For all of them we
have $\alpha = 1/2$, but the average densities of the corresponding graph-directed sets are likely to be different, which would
imply that the substitution dynamical systems associated with $\sig_k$ are pairwise non-isomorphic.

(2) In general, symmetric and one-sided average densities need not be equal, except in symmetric cases, such as the middle-thirds Cantor set.
\end{remark}

As a corollary of  Theorem \ref{TheoremSecondOrderSubstitutions}, we establish that almost every sequence in $X_\sigma$ admits an ``$\alpha$-dimensional frequency''.

\begin{corollary}\label{CorollaryD_Density}  Let $(X_\sigma,\nu,S)$ be a substitution system satisfying the assumptions of Theorem \ref{TheoremSecondOrderSubstitutions}. Then for every letter $b\in \mathcal B$,
the limit

$$ \lim\limits_{n\to\infty}\frac{1}{\log(n)} \sum_{1\leq k \leq n,\; x_k = b}\frac{1}{k^\alpha}$$
exists and  equals to $\alpha \cdot c \cdot \nu([b])$ for $\nu$-a.e. $x = (x_k)\in X_\sigma$.
\end{corollary}
\begin{proof} We will use  the same notation as in the proof of Theorem \ref{TheoremSecondOrderSubstitutions}. Fix a letter $b\in \mathcal B$.  Consider the function $f : X_\sigma \rightarrow \mathbb R$ such that $f(x) = 1$ if $x_0 = b$ and $f(x) = 0$ otherwise. Given a sequence $x\in X_\sigma$, denote by
$\ell_b(x,k) = \ell_b(x[1,k])$ the number of occurrences of the symbol $b$ in the word $x[1,k]$.
 Theorem \ref{TheoremSecondOrderSubstitutions} implies that
 \begin{equation}\label{EquationLetterDensity1}\lim_{n\to\infty}\frac{1}{\log(n)}\sum_{k=1}^\infty \frac{\ell_b(x,k)}{k^{\alpha+1}} = c \nu([b])\end{equation} for $\nu$-a.e. $x\in X_\sigma$. Fix a sequence $x\in X_\sigma$ satisfying Equation (\ref{EquationLetterDensity1}). Using  summation by parts, we get
  $$ \sum_{k=1}^n\frac{f(S^kx)}{k^\alpha} =  \sum_{k=1}^{n-1}\ell_b(x,k) \left(\frac{1}{k^\alpha} - \frac{1}{(k+1)^\alpha} \right) + \frac{\ell_b(x,n)}{n^\alpha}.$$
  Lemmas \ref{LemmaWordTilingLength} and \ref{LemmaBoundBetaWords} imply that $\ell_b(x,n)/\alpha^d$ is uniformly bounded in $n$.  Notice that
    $\left(\frac{1}{k^\alpha} - \frac{1}{(k+1)^\alpha} \right) \approx \frac{\alpha}{k^{\alpha+1}}.$  Thus, (\ref{EquationLetterDensity1}) yields
    \begin{eqnarray*}
\lim\limits_{n\to\infty}\frac{1}{\log(n)}\sum_{k=1}^n \frac{f(S^kx)}{k^{\alpha}} &  = &
\lim\limits_{n\to\infty} \frac{1}{\log(n)} \sum_{k=1}^n \ell_b(x,k)\frac{d}{k^{\alpha+1}} \\ & = & \alpha\int fd\nu = \alpha\cdot c \cdot \nu([b])
\end{eqnarray*}
    for $\nu$-a.e. $x\in X_\sigma$. \end{proof}

This may be compared with a result of \cite{Bell:2008}, which implies that all ``morphic'' sequences $x$ have
the {\it logarithmic frequency} of  letters. This means that for $a\in \mathcal A$ the following limit exists:
$$\lim\limits_{n\to\infty}\frac{1}{\log(n)}\sum_{k=1,\;x_k = a}^n\frac{1}{k}\,.$$
For our substitutions it is immediate that the logarithmic frequency of $b\in \B$ is zero for all $x\in X_\sig$, since already the ordinary frequency
$\lim_{n\to \infty} \frac{1}{n} \#\{k\le n:\ x_k=b\}$ equals zero.

%%%%%%%%%%%%%%
%
% EXAMPLES
\section{Examples and open questions}\label{SectionExamples}

In this section we consider a few  examples of  tiling and substitution systems and determine the parameter $\alpha$ appearing in the second order ergodic theorem.

\begin{example}
(``Cantor'' substitution).
Let $\mathcal A = \{0,1\}$ and $\sigma(0) = 000$, $\sigma(1) = 101$.
Consider the substitution system $(X_\sigma,S)$ associated to $\sigma$.  Observe that the graph-directed set of $\sigma$ is the middle-thirds Cantor set.
The dynamical system $(X_\sigma,S)$  admits a unique ergodic measure $\mu$ on $X_\sigma$ with the property $\mu([1]) = 1$. Then  Theorem \ref{TheoremSecondOrderSubstitutions} holds for the system  $(X_\sigma,S)$ with parameters $\alpha = \log(2)/\log(3)$ and $c>0$, where $c$ is the right average density of the Hausdorff measure on the middle-thirds Cantor set.
Actually, in this case the right and left average densities are equal to the symmetric average density; its numerical value has been computed
in \cite{PaZa}.
The second order ergodic theorem for the system $(X_\sigma,S)$  was originally established by A.~Fisher \cite{Fisher:1992}.
\end{example}

\begin{example}
Returning to Example~\ref{ex-carpet}, we see that
the associated graph-directed set is the classical Sierpi\'nski carpet.
 So Theorem \ref{TheoremMainResult} with parameter $\alpha = \log(8)/\log(3)$  applies to this tiling system.
\end{example}

%\begin{remark}
%In the special case when $\B$ contains only one tile type, the graph-directed IFS reduces to the more familiar ``standard" IFS, and the set $K_Q$ is the usual self-similar set, attractor of the IFS.
%\end{remark}

\begin{example}  Here we consider  a tiling dynamical systems with prototiles having fractal boundaries. Our example is  a modification of the system described in \cite[Section 7.2]{Solomyak:1997}, belonging to the family of  tilings constructed in \cite[Section 6]{Kenyon:1996}.

Let $r\approx .34115 + 1.1616 i$ be a root of the equation $x^3+x+1 = 0$. Let $T_a$, $T_b$, and $T_c$ be sets (prototiles) as described in Lemma 7.7 of \cite{Solomyak:1997}. We use the same notation as in \cite{Solomyak:1997}.

We note that $T_a$, $T_b$, and $T_c$ are compact subsets of $\mathbb C$.
Set $\theta(z) = r z$. Then $\theta(T_a) = T_b$; $\theta(T_b)$ is the union of a translation of $T_b$ and  of $T_c$; and $\theta(T_c)$  is a translation of $T_a$. This subdivision rule uniquely determines  a tile substitution $\Theta$.

 Set $\mathcal G = \Theta^2$. Thus, $\mathcal G(T_a)$ is the union of a translation of $T_b$ and of $T_c$; $\mathcal G(T_b)$ is the union  of  a translation of $T_b$, of $T_c$, and of $T_a$; and $\mathcal G(T_c)$ is a copy of $T_b$. Set $\varphi(z) = r^2z$. Then after the ``realification'' of $\mathbb C$, the map $\varphi$ can be represented as $\varphi((x,y)^T) = \lambda \cdot O\cdot (x,y)^T$, where $\lambda = |r^2|$ and  $$O = \left(
                         \begin{array}{cc}
                           \beta & -\gamma \\
                           \gamma & \beta \\
                         \end{array}
                       \right)
$$ with $\beta + \gamma i = r^2/|r^2|$.  Then $\lambda$ is the expansion constant of $\varphi$.

Assuming that the tiles $T_a$, $T_b$, and $T_c$ are colored  white,  denote by $S_a$, $S_b$, and $S_c$ their respective copies colored black. Extend the tile substitution $\mathcal G$ on $\{S_a,S_b,S_c\}$ as follows: $S_a$ is mapped into a union of $S_b$ and $S_c$ in the same way as $T_a$; $S_c$ is mapped into a copy of $S_b$ as $T_c$;  and the tile $S_b$ is mapped into a union  of $S_a$, $S_b$, and $S_c$ exactly as $T_b$, but with the tile $S_b$ being  replaced by the tile $T_b$ (of the same shape but of a different color).

\begin{sloppypar}
Denote by $\mathcal A$ the set of prototiles $\{T_a,T_b,T_c,S_a,S_b,S_c\}$. Consider the tiling dynamical system $(\Omega_\mathcal G,\mathbb R^2)$ associated to $\mathcal G : \mathcal A\rightarrow \mathcal A^+$. This system has a  unique minimal component determined by white tiles, see \cite[Lemma 2.10]{CortezSolomyak:2011}. The minimal component is non-periodic, see \cite[Section 7.2]{Solomyak:1997}. Hence, the substitution $\mathcal G$ satisfies all the conditions of Theorem \ref{TheoremInfiniteInvariantMeasures} yielding that this system admits a unique ``natural'' infinite invariant measure $\mu$ up to scaling.
The substitution matrix $M_\Gk$ is given by $$\left(
                                                        \begin{array}{cccccc}
                                                          0 & 1 & 0 & 0 & 0 & 0 \\
                                                          1 & 1 & 1 & 0 & 1 & 0 \\
                                                          1 & 1 & 0 & 0 & 0 & 0 \\
                                                          0 & 0 & 0 & 0 & 1 & 0 \\
                                                          0 & 0 & 0 & 1 & 0 & 1 \\
                                                          0 & 0 & 0 & 1 & 1 & 0 \\
                                                        \end{array}
                                                      \right)
$$
It follows that the matrix $B$ is equal to $\left(
                                               \begin{array}{ccc}
                                                 0 & 1 & 0 \\
                                                 1 & 0 & 1 \\
                                                 1 & 1 & 0 \\
                                               \end{array}
                                             \right).
$
Since $\rho(B)\approx 1.618$ and the expansion constant of $\mathcal G$ is $\lambda \approx 1.466$, we get that $\alpha = \log(\rho(B))/\log(\lambda) >1$. Hence, the second order theorem (Theorem \ref{TheoremMainResult}) with parameter $\alpha\approx 1.258$ applies to the system $(\Omega_\mathcal G,\mathbb R^2)$.
\end{sloppypar}
\end{example}

\begin{example}

This example is a non-minimal extension of the well-known Rauzy tiling \cite{Rauzy}.
We start with the Rauzy tiling itself.
Let $r\approx -0.7771845+1.11514i$ be the complex root of the equation $1-r-r^2-r^3=0$.
The tiles may be described using digit expansions in the base of $r$.
There are three prototiles $T_a$, $T_b$, and $T_c$, which may be represented as follows:
$$
T: = \Bigl\{\sum_{n=0}^\infty a_n r^{-n}:\ a_n\in \{0,1\},\ a_n a_{n+1} a_{n+2} \ne 111\ \mbox{for all} \ n\Bigr\}.
$$
Then
$$T_a:= r^{-1} T,\ T_b:= 1 + r^{-2} T,\ T_c:= 1+ r^{-1} + r^{-3} T.
$$
Clearly, $rT_a = T_a \cup T_b \cup T_c$, $r T_b = r + T_a$, and $r T_c = r+ T_b$. This determines the substitution rule. (Strictly speaking,
these prototiles do not satisfy our definition, since $T_b$ and $T_c$ do not contain the origin in the interior of their support, but this can
be easily rectified, translating the tiles. However, the given form of the tiles is more convenient.) All the tiles of the Rauzy tiling can
also be described using base $r$ expansions: for any {\em finite} sum $z=\sum_{n=-N}^{-1} a_n r^{-n}$ with the property that
$a_n\in \{0,1\},\ a_n a_{n+1} a_{n+2} \ne 111$ for all $n$, we have $z+T_a\in \T$ in all cases, $z+T_b\in \T$ iff $a_{-2}a_{-1}\ne 11$, and
$z+T_c\in \T$ iff $a_{-1}\ne 1$.

Now consider the ``extended'' tiling system, with the prototiles $T_a, T_b, T_c$ and $S_a$, $S_b$, which have the same support as
$T_a, T_b$ respectively, but have a different color (label). The substitution acts as before on $T_a, T_b, T_c$, and
$$
r S_a = S_a \cup S_b \cup T_c,\ \ rS_b = 1 + S_a.
$$
The matrix of the substitution is
$
M_\Gk =\left(
  \begin{array}{cc}
    A & C \\
   0 & B
  \end{array}
\right)$, where
$$
A = \left( \begin{array}{ccc} 1 & 1 & 0 \\ 1 & 0 & 1 \\ 1 & 0 & 0 \end{array} \right),\ \ \
B = \left( \begin{array}{cc} 1 & 1  \\ 1 & 0  \end{array} \right),\ \ \
C= \left( \begin{array}{cc} 0 & 0 \\ 0 & 0   \\ 1 & 0  \end{array} \right).
$$
The expansion $\lam=|r|\approx 1.3562$ is the same as for the Rauzy tiling, and $\rho(B) \approx 1.618$ is the golden ratio.
All the assumptions from Section~{SectionAssumptions} are easily verified. We get $\alpha = \log(\rho(B))/\log(\lam) \approx 1.57935>1$,
so Theorem~\ref{TheoremMainResult} applies. Figure 1 shows the ``cantorization'' of the tiling, so it gives an idea of both ``large-scale''
structure of the tiling and the ``small-scale'' structure of the graph-directed sets.

%\begin{figure}
%\includegraphics[width=120mm]{Rauzy2.png}
%\caption{``Cantorization'' of the tiling}
%\end{figure}
\end{example}

It is interesting to note that the ``cantorization'' of the tiling has a simple description using base $r$ expansions:
instead of all expansions using the digits $(a_n)$ with $111$ forbidden, one should consider all expansion with the sequence of digits
from the ``golden mean'' shift, that is, $11$ is
forbidden.

%\begin{itemize}
%\item Analogy between the tiling translation action and substitution action on one hand, and the horocycle/geodesic flows on manifolds of negative curvature
%has been repeatedly pointed out. Our non-primitive case corresponds, roughly, to the case of a non-compact, but geometrically finite manifold.
%
%\item The one-sided dynamical system $(\Om_0,\G^n)_{n\ge 0}$  is essentially the same as the scenery flow on the graph-directed sets, studied by a number of authors. Thus, the 2-sided system $(\Om_0,\G^n)_{n\in\Z}$ from Section 4 may be viewed as the natural extension of the scenery flow.

%\item It should also be possible to draw a connection with Furstenberg's CP-processes.
%\end{itemize}

\subsection{Open questions}

1. We had to impose some technical conditions on the substitution to prove the second order ergodic theorem. For example, we do not know if it holds for
the following substitution on $\{0,1,2\}$:
$$
0\mapsto 00000,\ \ 1\mapsto 1111,\ \ 2\mapsto 20212.
$$
The matrix of the substitution is
$M_\sig = \left(\begin{array}{ccc} 5 & 0 & 1 \\ 0 & 4 & 1 \\ 0 & 0 & 3\end{array}\right)$. Thus, there is an infinite ($\sig$-finite) invariant measure positive
and finite on cylinder sets containing $2$, however, there is no positive left eigenvector, so our methods do not work.

\smallskip

2. We proved that (in appropriate contexts) converge the logarithmic averages of the expressions
$$
R^{-\alpha}\int_{B_R}  f(\T-u)\,du\ \ \mbox{and}\ \ \ k^{-\alpha} \sum_{i=0}^{k-1} f(S^i x).
$$
But one can also view them as random variables (with $\T$ or $x$ taken randomly from the substitution space, according to the invariant measure normalized
on the appropriate cylinder set), and inquire whether they converge in distribution as $R$ (resp.\ $k$) tend to infinity along a subsequence? For instance, it
appears that for the ``integer Cantor'' substitution from Example 1, we get that
$
2^{-n} \sum_{i=0}^{3^n-1} f(S^i x)
$
tends to the uniform distribution on $[0,1]$ as $n\to \infty$, for $f$ the characteristic function of $[1]$ (and then for all $f\in L^1(X_\sigma,\mu)$ with $\int f\,d\mu=1$).

\medskip

\noindent {\bf Acknowledgment.} We are grateful to Karl Petersen, who asked whether the second-order ergodic theorem holds for the Sierpi\'nski gasket tiling
system and who told us about  A. Fisher's paper \cite{Fisher:1992}. We would also like to thank the referee for his/her valuable comments.

%%%%%%%%%%%%%%%%%%%%%%%%%%%%%%%%%%%%%%%%%%%%%%%%%%%
%
%

%%%%%%%%%%%%%%%%%%%%%%%%%%%%%%%%%%%%%%%%%

%%%%%%%%%%%%%%%%%%%%%%%
%
%


\begin{thebibliography}{ZZZ}

\bibitem[A]{Aaronson:book} J.~Aaronson, \emph{An introduction to infinite ergodic theory}. Mathematical Surveys and Monographs, 50. American Mathematical Society, Providence, RI, 1997.

\bibitem[ADF]{Aarson_Denker_Fisher:1992} J.~Aaronson, M.~Denker, and A.~Fisher, \emph{Second order ergodic theorems for ergodic transformations of infinite measure spaces}. Proc. Amer. Math. Soc. {\bf 114} (1992), no. 1, 115--127.

\bibitem[AP]{Anderson_Putnam:1998}    J.~Anderson and I.~Putnam, \emph{Topological invariants for substitution tilings and their associated $C^*$-algebras}. Ergodic Theory Dynam. Systems \textbf{18} (1998), no. 3, 509--537.

\bibitem[Beck]{Becker:1983} M. Becker, {\it A ratio ergodic theorem for groups of measure-preserving transformations}. Illinois J.\ Math. {\bf 27} (1983), 562--570.

\bibitem[BF]{BeFi} T. Bedford and A. Fisher, {\it Analogues of the Lebesgue density theorem for fractal sets of reals and integers}. Proc.\ London Math.\ Soc. (3) {\bf 64} (1992), 95--124.

\bibitem[Be]{Bell:2008} J.~Bell, {\it Logarithmic frequency in morphic sequences}. J. Theor. Nombres Bordeaux {\bf 20} (2008), no. 2, 227--241.

\bibitem[BKM]{BezuglyiKwiatkowskiMedynets:2009} S.~Bezuglyi, J.~Kwiatkowski, K.~Medynets,  \emph{Aperiodic substitution systems and their Bratteli diagrams}. Ergodic Theory Dynam. Systems \textbf{29} (2009), no. 1, 37--72.

\bibitem[BKMS]{BezuglyiKwiatkowskiMedynetsSolomyak:2010}  S.~Bezuglyi, J.~Kwiatkowski, K.~Medynets, B.~Solomyak, \emph{Invariant measures on stationary Bratteli diagrams}. Ergodic Theory Dynam. Systems \textbf{30} (2010), no. 4, 973--1007.

%\bibitem[BKMS2]{BezuglyiKwiatkowskiMedynetsSolomyak:FiniteRank}  S.~Bezuglyi, J.~Kwiatkowski, K.~Medynets, B.~Solomyak, \emph{Finite Rank Bratteli diagrams: Structure of Invariant Measures},  to appear in Trans. AMS (2012), arXiv:1003.2816v5.

\bibitem[CS]{CortezSolomyak:2011} M.I.~Cortez and B.~Solomyak, \emph{Invariant measures for non-primitive tiling substitutions}.  J. Anal. Math. \textbf{115} (2011), 293--342.

\bibitem[D]{Danzer} L. Danzer, {\em Inflation species of planar tilings which are not of locally finite complexity}.  Proc. Steklov Inst. Math. {\bf 239} (2002), 108--116.

\bibitem[E]{Edgar:1997} G.~Edgar, {\em Integral, Probability, and Fractal Measures}. Springer, 1997.

\bibitem[Fa]{Falconer:Techniques} K.~Falconer, \emph{Techniques in fractal geometry}.  John Wiley \& Sons, Ltd., Chichester, 1997.

%\bibitem[Fe]{Federer:book} H.~Federer, \emph{Geometric measure theory}. Die Grundlehren der mathematischen Wissenschaften, Band 153 Springer-Verlag New York Inc., New York 1969 xiv+676 pp.

%\bibitem[Fi1]{Fisher:1987} A.~Fisher, \emph{Convex-invariant means and a pathwise central limit theorem}. Adv. in Math. \textbf{63} (1987), no. 3, 213 -- 246.

\bibitem[Fi1]{Fisher:1992} A.~Fisher, \emph{Integer Cantor sets and an order-two ergodic theorem}. Ergodic Theory Dynam. Systems \textbf{13} (1993), no. 1, 45--64.

\bibitem[FR]{FraRobi} N. P. Frank and E. A. Robinson, Jr., {\em Generalized $\beta$-expansions, substitution tilings, and local finiteness}. Trans. Amer. Math. Soc. {\bf 360} (2008), 1163--1177.

\bibitem[G]{Gant} F. R. Gantmacher, {\em Applications of the theory of martrices}. Interscience Publishers, Inc., New York, 1959.

\bibitem[HY]{HamaYuasa:2011} M. Hama and H. Yuasa,
{\em Invariant measures for subshifts arising from substitutions of some primitive components}.
Hokkaido Math. J. {\bf 40} (2011), no. 2, 279--312.

\bibitem[H]{Hochman:2010} M.~Hochman, \emph{A ratio ergodic theorem for multiparameter non-singular actions}. J. Eur. Math. Soc.  \textbf{12} (2010), no. 2,
365--383.

\bibitem[HuW]{hurewicz_wallman} W.~Hurewicz, and H.~Wallman, \emph{Dimension Theory}. Princeton Mathematical Series, v. 4. Princeton University Press, Princeton, N. J., 1941.

\bibitem[K1]{Kenyon:1991}     R.~Kenyon, \emph{Self-replicating tilings. Symbolic dynamics and its applications} (New Haven, CT, 1991), 239-263, Contemp. Math., \textbf{135}, Amer. Math. Soc., Providence, RI, 1992.

\bibitem[K2]{Kenyon:1996}    R.~Kenyon, \emph{The construction of self-similar tilings}. Geom. Funct. Anal. \textbf{6} (1996), no. 3, 471--488.

\bibitem[LS]{Ledrappier_Sarig:2008} F.~Ledrappier and O.~Sarig, \emph{Fluctuations of ergodic sums for horocycle flows on $\mathbb Z^d$-covers of finite volume surfaces}. Discrete Contin. Dyn. Syst. \textbf{22} (2008), no. 1-2, 247--325.

%\bibitem[LM]{LindMarcus:book} D. Lind and B. Marcus, {\em An Introduction to Symbolic Dynamics and Coding}. Cambridge University Press, Cambridge, 1995. xvi+495 pp.

\bibitem[MW]{MauldinWilliams:1988} R.D.~Mauldin, S.C.~Williams, \emph{Hausdorff dimension in graph directed constructions}. Trans. Amer. Math. Soc. \textbf{309} (1988), no. 2, 811--829.

\bibitem[PZ]{PaZa} N. Patzschke and  M. Z\"ahle, M. {\em Fractional differentiation in the self-affine case. III. The density of the Cantor set.} Proc. Amer. Math. Soc. {\bf 117} (1993), no. 1, 137--144.

\bibitem[P]{prag} B. Praggastis, {\em Numeration systems and Markov partitions from self similar tilings}. {Trans.\ Amer.\ Math.\ Soc.} {\bf 351} (1999), no. 8, 3315--3349.


\bibitem[RW]{RadinWolff} C.~Radin, M.~Wolff, \emph{Space tilings and local isomorphism}. Geom. Dedicata \textbf{42} (1992), no. 3, 355--360.

\bibitem[R]{Rauzy} G. Rauzy, {\em Nombres alg\'ebriques et substitutions.} Bull.\ Soc. Math.\ France {\bf 110}(2) (1982), 147--178.

\bibitem[Ro]{Robi}  E. A. Robinson, Jr.,
{\em Symbolic dynamics and tilings of $\R^d$}.
Symbolic dynamics and its applications, Proc.\ Sympos.\ Appl.\ Math.,
Vol.\ 60, Amer.\ Math.\ Soc., Providence, RI, 2004, pp.\ 81--119.

\bibitem[S]{Schneider} H.~Schneider, {\it The influence of the marked reduced graph of a nonnegative matrix on the Jordan form and on related properties: a survey.} Proceedings of the symposium on operator theory (Athens, 1985). Linear Algebra Appl. \textbf{84} (1986), 161--189.

\bibitem[So1]{Solomyak:1997} B.~Solomyak, \emph{Dynamics of self-similar tilings}. Ergodic Theory Dynam. Systems \textbf{17} (1997), no. 3, 695--738.

\bibitem[So2]{Solomyak:1998} B. Solomyak,  {\it  Nonperiodicity implies unique composition for self-similar translationally finite tilings.}  Discrete Comput. Geom. {\bf 20} (1998), 265--279.


%\bibitem[TS]{Tam-Schneider} B.T.~Tam, H.~Schneider, \emph{On the invariant faces associated with a cone-preserving map}. Trans. Amer. Math. Soc.
%\textbf{353} (2001), no. 1, 209 -- 245.

\bibitem[Y]{Yuasa:2007}  H.~Yuasa, {\em Invariant measures for the subshifts arising from non-primitive substitutions}.  J. D'Analyse Math. {\bf 102} (2007), 143--180.


\end{thebibliography}
\end{document}